\documentclass[12pt,epsfig,amsfonts]{amsart} 
\usepackage{amsmath,amsthm,amssymb,amscd,epsfig,color}
\usepackage{graphicx}
\usepackage{ulem}
\usepackage{mathrsfs}
\newtheorem*{theorema}{Theorem A}
\newtheorem*{theoremb}{Theorem B}

\newtheorem{prop}{Proposition}[section]
\newtheorem{lemma}[prop]{Lemma}

\newtheorem{sublemma}[prop]{Sublemma}
\newtheorem{theorem}[prop]{Theorem}
\newtheorem{cor}[prop]{Corollary}

\theoremstyle{remark}
\newtheorem{remark}[prop]{Remark}

\numberwithin{equation}{section}

\begin{document}

\author{Johannes Jaerisch and Hiroki Takahasi}

\address{Graduate School of Mathematics, Nagoya University,
Furocho, Chikusaku, Nagoya, 464-8602, JAPAN} 
\email{jaerisch@math.nagoya-u.ac.jp}
%\urladdr{http://www.math.nagoya-u.ac.jp/~jaerisch/}

\address{Keio Institute of Pure and Applied Sciences (KiPAS), Department of Mathematics,
Keio University, Yokohama,
223-8522, JAPAN} 
\email{hiroki@math.keio.ac.jp}
%\urladdr{\texttt{https://ht.math.keio.ac.jp/}}


\title[Large deviations of homological growth rates]{
Large deviations
of
homological growth rates\\
for hyperbolic surfaces} 
 %\date{\today}
 \maketitle
 \begin{abstract} 
 We perform a large deviations analysis of homological growth rates of oriented geodesics on hyperbolic surfaces. 
For surfaces uniformized by  a wide class of Fuchsian groups of the first kind,
 we prove the existence of
 the rate function which estimates 
 exponential 
 probabilities   with which the homological growth rates
stay away from the mean value. 
 The rate function is given in terms of the multifractal dimension spectrum described in our earlier result \cite{JaeTak22}.
%\sout{[arXiv:2204.08907].}
We also establish an  Erd\H{o}s-R\'enyi law, and refined large deviations upper bounds.
 \end{abstract}
 %Statements as in (a) to (e)
%\textcolor{red}{(To be removed?)
%Our notion of generalized Poincar\'e exponent is closely related to the Manhatten curve associated with pairs of discrete subgroups of the isometry group of  rank 1 symmetric spaces \cite{Burger93}.   
%For convex cocompact groups, Pollicott and Sharp \cite{PS2016} have established the real-analyticity of the Manhatten curve. Recently, this result was extended to certain Fuchsian groups with cusps \cite{Kao2020}.
%}

 \section{Introduction}

 Let $G$ denote a  finitely generated, non-elementary 
% \textcolor{cyan}{\cite[p.180]{Rat94}}
Fuchsian group of the first kind
 acting on the Poincar\'e disc model $(\mathbb D,d)$ of hyperbolic space. 
% Having fixed a finite set of generators of $G$, 
For each $g\in G$ let $|g|$ denote the minimal number of elements in a fixed set of generators needed to represent $g$, called the word length of $g$. 
% It follows from the triangle inequality that
%Clearly, 
There exists  $\overline{\alpha} >0$ such that $d(0,g0)\le \overline{\alpha} |g|$ for all $g\in G$. 
If $\mathbb D/G$ is compact, the Milnor-Swarc Lemma implies the existence of  $\underline\alpha>0$ such that $ d(0,g0)\geq\underline\alpha |g| $
for all $g\in G$. If $\mathbb{D}/G$  is non-compact, there exists $C>0$ such that $d(0,g0)\geq
 2\log|g|-C$ for all $g\in G$ \cite{floyd80}. 
 %\textcolor{red}{(jj: remove?  we never refer to (1.1) and (1.2). Also, the upper bounds are trivial.) The following general bounds relating word length and displacement are well known \cite[Section~4]{floyd80}. If $G$ has no parabolic element,  there exist $\alpha_1, \alpha_2 >0$ such that for all $g\in G$,\begin{equation}\label{growth1}\alpha_1 |g| \le d(0,g0)\le \alpha_2 |g|.\end{equation}
%where $|g|$ denotes the minimal number of elements of $G_R$ needed to represent $g$, called the word length of $g$. If $G$ has a parabolic element, there exist  $C>0,\alpha_3>0$ such that for all $g\in G$, \begin{equation}\label{growth2}2\log|g|-C \le d(0,g0)\leq\alpha_3|g| .\end{equation} }
The complexity of the  
 action of $G$ is reflected in the fact that
 the  growth rate of $d(0,g0)/|g|$ as $|g|\rightarrow \infty$  takes on uncountably many values, and rates of convergence are not uniform. In this paper 
 we perform a large deviations analysis of this growth rate  along oriented geodesics.

 Let $R\subset \mathbb{D}$ be a convex, locally finite fundamental domain for $G$ which contains $0$ in its interior \cite{Bea83}. The finite set of side-pairings of $R$, denoted by $G_R$, defines a symmetric set of generators of $G$. 
 We assume $R$ is {\it admissible}, see Section~\ref{sec:cutting} for the definition.
% We call $R$ {\it admissible} if it has {\it even corners} \cite{BowSer79,Ser86} and satisfies a technical condition. We refer the reader to Section~\ref{sec:cutting} for the details.  
 Let $\mathscr{R}$ denote the set of oriented complete geodesics  joining two points  in $\mathbb S^1$ and intersecting the interior of $R$. If $\gamma \in \mathscr{R}$ cuts %through 
 successively the copies $R, g_0R, g_0g_1R, \ldots $ of $R$, with $g_i \in G_R$ for $i=0,1,\ldots\in\mathbb N$, then $g_{0}, g_{1}, g_{2}, \ldots$
  is called the {\it  cutting sequence} of $\gamma$ (see \textsc{Figure}~\ref{fig:geo}). 
%  \textcolor{red}{(remove) By slightly perturbing geodesics through a vertex of $R$,  we will define for each $\gamma\in \mathscr{R}$ a unique finite or infinite cutting sequence in  Section~\ref{sec:cutting}.   For $\gamma \in \mathscr{R}$ with the cutting sequence    $g_{0},g_{1},\ldots$   of length at least $n\ge 1$,  we define    \[t_{n}(\gamma)=d(0,g_0 g_1 \cdots g_{n-1}0),\]  and we call  $t_n(\gamma)/n$ the homological growth rate  of $\gamma$ \cite{KesStr04}.  Since $R$ has even corners,    $g_0\cdots g_{n-1}$ has word length $n$ with respect to $G_R$ \cite[Theorem~3.1(ii)]{Ser86}.  We denote by $\Lambda_c=\Lambda_c(G)$ the conical limit set of $G$, which  equals the complement of the countable set of  fixed points of parabolic elements of $G$ \cite{BeaMask74}.  A geodesic $\gamma \in \mathscr{R}$  has the infinite cutting sequence  if and only if its positive endpoint  $\gamma^+ $  belongs to $\Lambda_c$ \cite[Lemma~2.2]{JaeTak22}.}
For each $\gamma \in \mathscr{R}$ whose positive endpoint $\gamma^+$ is contained in the conical limit set $\Lambda_c=\Lambda_c(G)$, an infinite cutting sequence $(g_n)_{n=0}^\infty$ will be uniquely defined in 
  Section~\ref{sec:cutting}. 
  %For $\gamma \in \mathscr{R}$ with the cutting sequence $(g_n)_{n=0}^\infty$ and
  For each $n\geq1$,
  the word length of $g_0\cdots g_{n-1}$ with respect to $G_R$ equals $n$ \cite[Theorem~3.1(ii)]{Ser86}. 
 We define
    \[t_{n}(\gamma)=d(0,g_0 g_1 \cdots g_{n-1}0),\] 
and call  $t_n(\gamma)/n$ {\it the homological growth rate} of $\gamma$. 
  
The set $\Lambda_c$
 equals the complement of the countable set of fixed points of parabolic elements of $G$ \cite{BeaMask74}, and
the limit of the homological growth rates takes on uncountably many values. We put \[\alpha_+=\sup_{\gamma\in\mathscr{R} ,\gamma^+ \in \Lambda_c}\limsup_{n\to\infty}\frac{t_n(\gamma)}{n}\quad
\text{and}\quad\alpha_-=\inf_{\gamma\in\mathscr{R} ,\gamma^+ \in \Lambda_c}\liminf_{n\to\infty}\frac{t_n(\gamma)}{n},\]
and for each $\alpha\in[\alpha_-,\alpha_+]$ define the {\it level set}
 \[\mathscr{H}(\alpha)=\left\{\xi\in\Lambda_c\colon\text{there exists $\gamma\in\mathscr{R}$%\cap\mathscr{A}$
  \ such that $\gamma^+=\xi$ and } \lim_{n\to\infty}\frac{t_n(\gamma)}{n}=\alpha\right\}.\] 

With a slight abuse of notation, let $|$ $\cdot$ $|$ denote the Lebesgue measure on $\mathbb S^1$.
In this paper we are concerned with the
sets 
%of conical limit points for which $t_n(\gamma)/n$ stay away from $\alpha_G$, namely, sets
\[\mathscr{H}_n(A)=\left\{\xi\in\Lambda_c\colon \exists\gamma\in\mathscr{R}\text{  s.t. }\gamma^+=\xi,\ \frac{t_n(\gamma)}{n}\in A\right\},\] 
where $A\subset\mathbb R$ and $n\geq1$, and
 interested in giving bounds on
$|\mathscr{H}_n(A)|$.   Put
    \[\underline{\alpha}=\inf_{g\in G\setminus\{1\}}\frac{d(0,g0)}{|g|}\ \text{ and }\ \overline{\alpha}=\sup_{g\in G }\frac{d(0,g0)}{|g|}.\]
    In fact
     we have $\underline{\alpha}=\alpha_-$ and $\overline{\alpha}=\alpha_+$
      (see Theorem~\ref{bslyapunov}). 
Hence, $\mathscr{H}_n(A)\neq\emptyset$ implies $A\cap[\underline\alpha,\overline\alpha]\neq\emptyset.$ By the Milnor-Swarc Lemma, $\underline\alpha>0$ holds if and only if $G$ has no parabolic element. 
      There exists a constant $\alpha_G\in[\underline\alpha,\overline\alpha)$ (see \eqref{ag} for the definition) such that
$\mathscr{H}(\alpha_G)$ has the full Lebesgue measure \cite[Proposition~A.8]{JaeTak22}.
Hence, if $A$ is a closed set not containing $\alpha_G$ then 
$\lim_{n\to\infty}|\mathscr{H}_n(A)|=0$.
Our result below asserts that
the speed of this convergence is exponential, and
the exponential rate is controlled by an analytic function.

\begin{figure}
\begin{center}
\includegraphics[height=5cm,width=10cm]{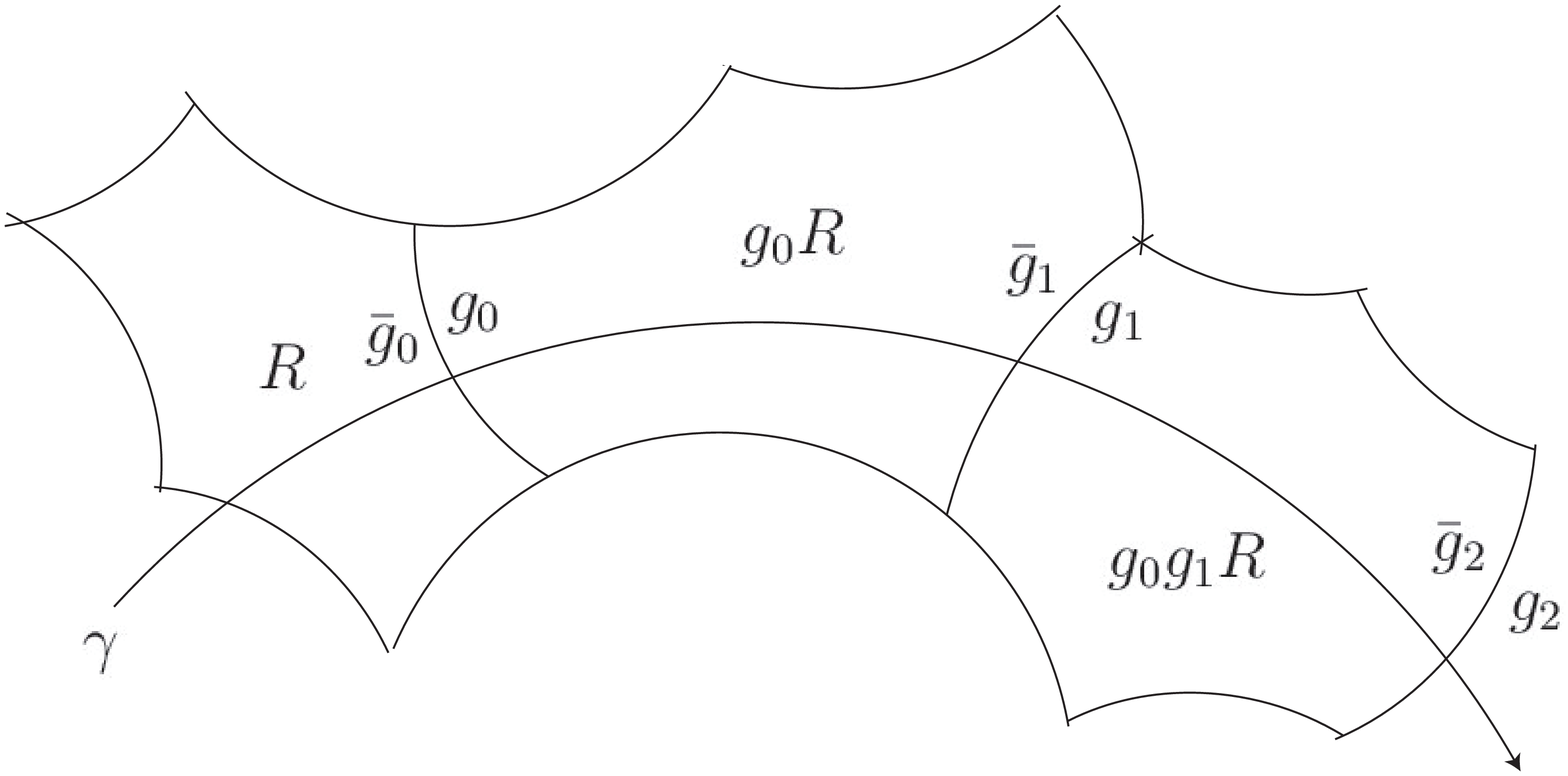}
\caption{An oriented geodesic $\gamma$ crossing the copies of the fundamental domain $R$.  }
\label{fig:geo}\end{center}
\end{figure}
    
  %Since $G$ is non-elementary, $\delta_G>0$.
%It is known that $\delta_G>0$ iff  $G$ is non-elementary.
 %($G$ is elementary, if it is generated by a single element).

\begin{theorema}%[Large deviation principle]
\label{ldp-thm}
Let $G$ be a finitely generated Fuchsian group of the first kind with an admissible fundamental domain. 
There exists a function $I\colon\mathbb R\to[0,+\infty]$  with the following properties:
\begin{itemize}
\item[(a)] For any non-degenerate interval $A$ 
intersecting $(\underline\alpha,\overline\alpha)$  we have
 \[\lim_{n\to\infty}\frac{1}{n}\log\left|
 \mathscr{H}_n(A)\right|=-\inf_{\alpha\in A }
I(\alpha).\]
 
\item[(b)] We have 
$I^{-1}(+\infty)=\mathbb R\setminus[\underline\alpha,\overline\alpha]$,
$I$ is continuous on $[\underline\alpha,\overline\alpha]$, analytic on $(\underline\alpha,\overline\alpha)$, 
 $I(\alpha_G)=0$
 and
$I''>0$ on $(\underline\alpha,\overline\alpha)$.
If $G$ has no parabolic element, then
$\lim_{\alpha\searrow\underline\alpha}I'(\alpha)=-\infty$ and $\lim_{\alpha\nearrow\overline\alpha}I'(\alpha)=+\infty$.
If $G$ has a parabolic element,
then $\lim_{\alpha\searrow\underline\alpha}I'(\alpha)=0$ and $\lim_{\alpha\nearrow\overline\alpha}I'(\alpha)=+\infty$.
\end{itemize}
\end{theorema}
%It is well known \cite{Bea71,Pat76} \cite[Corollary~26]{Sul79} that  $\delta_G$ is equal to the {\it  Poincar\'e exponent of $G$} given by
%\[\inf \left\{\beta\geq0  \colon \sum_{g\in G} \exp(-\beta d(0,g0))<+\infty\right\}.\]
%   It is 
% well-known \cite{Bea71,Pat76} \cite[Corollary~26]{Sul79} that the Hausdorff dimension $\delta_G$ \textcolor{blue}{(JJ: should we really *define* $\delta_G$ as Hausdorff dimension? Better: Define $\delta_G$ as critical exponent and then state $\dim_{\rm H} \Lambda(G) =\delta (G)$)? 

A function on $\mathbb R$ with the properties in Theorem~A is called a {\it rate function}. Clearly, the rate function is unique. It is tightly linked to the multifractal dimension spectrum of homological growth rates analyzed in \cite{JaeTak22}.
Let $\dim_{\rm H}$ denote the Hausdorff dimension on $\mathbb S^1$, and for $\alpha\in[\underline\alpha,\overline\alpha]$ we set
    \[b(\alpha)=\dim_{\rm H} \mathscr H(\alpha) .\]
     The function
 $\alpha\mapsto b(\alpha)$ is called 
 the {\it $\mathscr{H}$-spectrum} \cite{JaeTak22}.
            As is evident from the proof of Theorem~A, the rate function $I$ (see \textsc{Figure}~\ref{rate-figure}) is given by
\begin{equation}\label{rate-f}I(\alpha)=\begin{cases}\alpha(1-b(\alpha)) &\text{ for }\alpha\in[\underline\alpha,\overline\alpha],\\
+\infty&\text{ for }\alpha\in\mathbb R\setminus[\underline\alpha,\overline\alpha].\end{cases}\end{equation}

For any finitely generated Fuchsian group $G$ of the first kind, Bowen and Series \cite{BowSer79} constructed a piecewise analytic Markov map $f\colon \mathbb{S}^1\to\mathbb S^1$ which is orbit equivalent to the action of $G$ on the limit set $\mathbb S^1$, now called the {\it  Bowen-Series map}. 
Using hyperbolic geometry and a result of Series in \cite{Ser86},
we show that
%if $R$ is an admissible fundamental domain, then
  $|t_n(\gamma)-\log |(f^n)'\gamma^+|$ is
 bounded from above by $2\log n+{\rm const.}$, see
  Propositions~\ref{lem-logdist-nopar} and \ref{lem-logdist}. This\footnote{From \cite[Propositions~2.8 and 2.9]{JaeTak22}, for any $\gamma\in\mathscr{R}$ we have $|t_n(\gamma)-\log |(f^n)'\gamma^+|\leq 2\log D_n+{\rm const.}$ where $\log D_n=o(n)$ $(n\to\infty)$, see Proposition~\ref{MILD}. This bound
  suffices for the proof of Theorem~A, but results in a weaker upper bound than that in Theorem~B(b). For details on $D_n$, see Remark~\ref{D_n-remark}.}
   reduces the proof of Theorem~A(a) to proving the level-1 Large Deviation Principle (LDP) for the Birkhoff averages of $-\log|f'|$.
   For a general account on large deviations, including the meanings of level-1 and level-2,
 we refer the reader to the book of Ellis \cite[Chapter~1]{Ell85}.

  Apart from special cases, the function $-\log|f'|$ has discontinuities. To `hide' them,
we represent $f$ as a symbolic dynamics
over a finite alphabet. We then show the level-1 LDP on this shift space (Proposition~\ref{ld-lem}), 
and that $I$ in \eqref{rate-f} is the rate function (Proposition~\ref{rate-equal}). 
To identify the rate function and verify its properties as in Theorem~A(b),  we use our earlier results in \cite{JaeTak22} on the multifractal analysis of the homological growth rates.

 The Markov partition for the Bowen-Series map $f$ constructed in \cite{BowSer79} is an infinite partition if and only if $G$ has a parabolic element.
If $G$ has no parabolic element, the Markov partition semiconjugates $f$
 to a transitive finite Markov shift. %and the conjugacy is H\"older continuous.
%There is a unique Gibbs-equilibrium state on the associated finite shift space, and 
The function $-\log|f'|$ induces a H\"older continuous function 
on this shift space, and 
the level-2 LDP holds 
\cite{Kif90,OrePel89,Tak84}.  
Via the contraction principle we obtain the desired level-1 LDP.
It also follows directly from the estimates in \cite{You90}.

If $G$ has a parabolic element, instead of
 the infinite Markov partition in \cite{BowSer79} we use the finite Markov partition
constructed in \cite{JaeTak22} 
that still semiconjugates $f$
 to a transitive finite Markov shift.
 Although the function induced from $-\log|f'|$ on this shift space is no longer H\"older continuous, 
 one can still extend
 arguments in \cite{Tak84} to establish
the level-2 LDP.
Via the contraction principle we obtain the desired level-1 LDP.

% If $G$ has a parabolic element, the corresponding Bowen-Series map has a neutral periodic point, and the unique invariant measure absolutely continuous with respect to Lebesgue is an infinite measure. Large deviations for this class of maps are not well-known.

\begin{figure}
\begin{center}
\includegraphics[height=5cm,width=13cm]{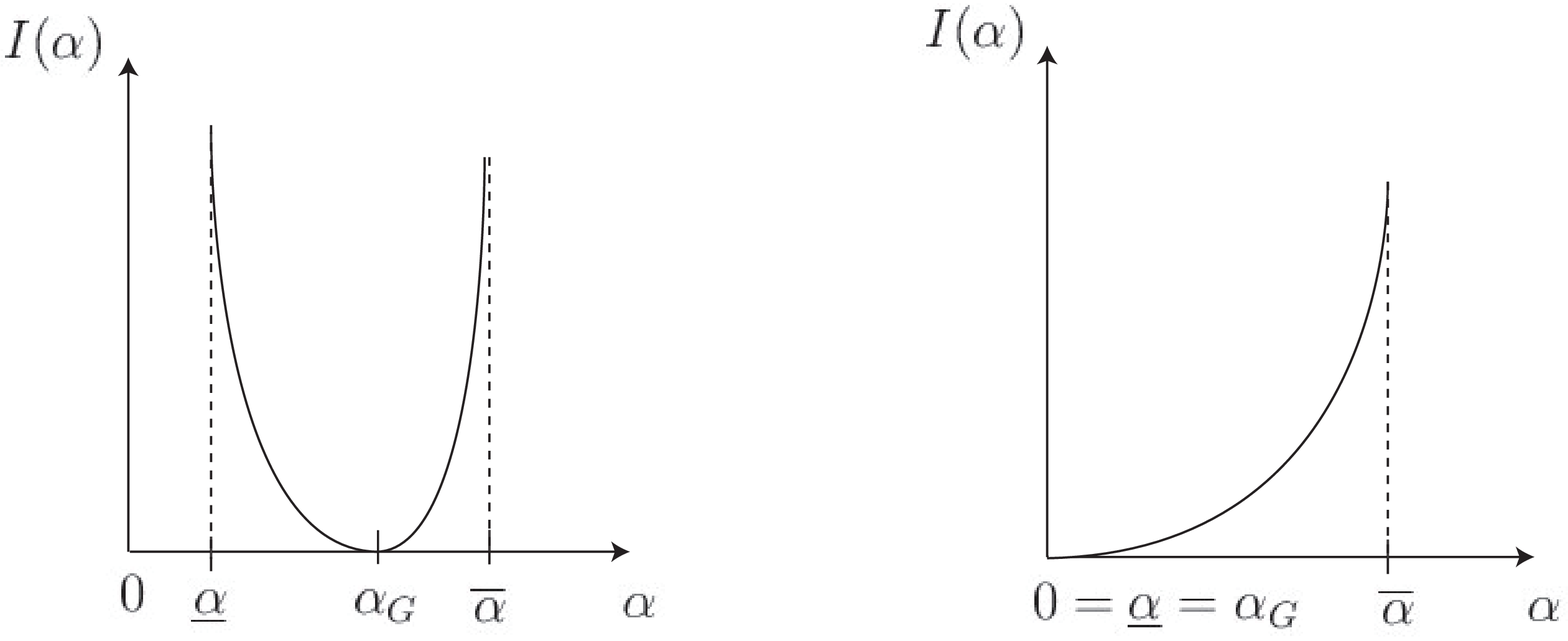}\end{center}
\caption{The graphs of the rate function $I$: %$\alpha\in[\alpha_-,\alpha_+]\mapsto I(\alpha)$: 
$G$ has no parabolic element (left);
$G$ has a parabolic element (right).}\label{rate-figure}\end{figure}

Theorem~A has a further interesting consequence.
If $G$ has no parabolic element, then
 $|t_n(\gamma)-\log |(f^n)'\gamma^+||$ is uniformly bounded (Proposition~\ref{lem-logdist}). Hence, combining Theorem~A with
\cite[Proposition~2.2]{DenNic13} we obtain the following statement.
\begin{cor}[An Erd\H{o}s-R\'enyi law]\label{ER-cor}
Let $G$ be a finitely generated Fuchsian group of the first kind without a parabolic element and with an admissible fundamental domain.
For any $\alpha\in(\underline\alpha,\overline\alpha)$, 
 for Lebesgue almost every $\xi\in\Lambda_c(G)$ and for any 
 $\gamma\in\mathscr{R}$ with $\gamma^+=\xi$ we have
\[
\lim_{n\to\infty}\max_{0\leq m\leq n-\lfloor\log n/I(\alpha)\rfloor}\frac{t_{m+\lfloor\log n/I(\alpha)\rfloor}(\gamma)-t_{m}(\gamma) }{\log n}=\frac{\alpha}{I(\alpha)}=\frac{1}{1-b(\alpha)}.\]
%\begin{itemize}
%\item[(a)] 
%}
%\item[(b)]If $\{a_n\}$ is a sequence of positive reals satisfying $a_n\leq n$ for all $n\geq1$ and $\lim_{n\to\infty}a_n/\log n=+\infty$, then for Lebesgue almost every $\xi\in\Lambda_c(G)$ and any   $\gamma\in\mathscr{R}$ with $\gamma^+=\xi$, \[\lim_{n\to\infty}\max_{0\leq m\leq n-\lfloor a_n\rfloor}\frac{t_{m+\lfloor a_n\rfloor}(\gamma)-t_{m}(\gamma) }{a_n}=\alpha_G.\]
%\end{itemize}
\end{cor}
The $\mathscr{H}$-spectrum $\alpha\in[\underline\alpha,\overline\alpha]\mapsto b(\alpha)$ encodes
information on the complexity of the limit set. 
Although determined by rare events in the sense of the Lebesgue measure on $\Lambda_c$,
the $\mathscr{H}$-spectrum can be computed from a single typical oriented geodesic as stated in Corollary~\ref{ER-cor}.
It is desirable to establish a similar formula for groups with parabolic elements.

Giving sharp bounds on $|\mathscr{H}_n(\cdot)|$ is more delicate. Suppose $G$ has no parabolic element.
Combining the results in
 \cite{ChaCol05} \cite[Korollar~5.8]{Kes99} with ours we obtain constants $0<c_0<c_1$ and $\varepsilon_0>0$ such that for $n\geq 1+|I'(\alpha)|^{-4}$,
% \textcolor{red}{$I'(\alpha)\rightarrow b''(\alpha)$ on RHS? [18] uploaded to overleaf.} \textcolor{blue}{everything is OK because $I'$ and $b''$ differ by constant factor.... remove comment OK}
\[\frac{c_0e^{-I(\alpha)n}}{-I'(\alpha)\sqrt{n}}\leq \mathscr{H}_n((-\infty,\alpha])\leq \frac{c_1e^{-I(\alpha)n}}{-I'(\alpha)\sqrt{n}}\ \text{ for }
\alpha\in(\alpha_G-\varepsilon_0,\alpha_G),\text{ and }\] 
\[\frac{c_0e^{-I(\alpha)n}}{I'(\alpha)\sqrt{n}}\leq\mathscr{H}_n([\alpha,+\infty))\leq \frac{c_1e^{-I(\alpha)n}}{I'(\alpha)\sqrt{n}}\ \text{ for }\alpha\in(\alpha_G,\alpha_G+\varepsilon_0),\]
which are
  in agreement with the case of the sum of independently identically distributed random variables \cite{BaRa60}.
  These sharp bounds were used in  \cite{ChaCol05} to obtain convergence rates in the Erd\H{o}s-R\'enyi law.
  
The methods in  \cite{ChaCol05,Kes99} rely on a perturbation theory of transfer operators, and so those $\alpha$ not close to $\alpha_G$
are unaccounted for.
We have
obtained upper bounds 
which are weaker than \cite{ChaCol05,Kes99}  for $\alpha$ close to $\alpha_G$ but valid for all $\alpha\in(\underline\alpha,\overline\alpha)\setminus\{0\}$.
For sets $A\subset\mathbb R$, $K\subset\mathbb D$ and $n\geq1$, let
\[\mathscr{H}_n(A,K)=\left\{\xi\in\Lambda_c\colon \exists\gamma\in\mathscr{R}\text{  s.t. }\gamma^+=\xi,\ \gamma\cap K\neq\emptyset,\ \frac{t_n(\gamma)}{n}\in A\right\}.\]
\begin{theoremb}
\label{ldp1}Let $G$ be a finitely generated Fuchsian group of the first kind with an admissible fundamental domain. 
\begin{itemize}
    \item[(a)] If $G$ has no parabolic element, then there exist constants $\kappa_0>0$, $\kappa_1>1$, $\kappa_2>1$ such that the following holds.
For any $\alpha\in(\alpha_G,\overline\alpha)$ and
 any $n\geq \kappa_0/\alpha$ we have
  \[\left|\mathscr{H}_n([\alpha,+\infty))\right| \leq 
 \kappa_1\kappa_2^{I'(\alpha)}e^{ -I(\alpha)n}.\]
 For any $\alpha\in(\underline\alpha,\alpha_G)$ and any $n\geq \kappa_0/\alpha$ we have
  \[\left|\mathscr{H}_n((-\infty,\alpha])\right|\leq 
\kappa_1\kappa_2^{-I'(\alpha)}e^{ -I(\alpha)n}.\]
\item[(b)]  %Suppose 
If $G$ has a parabolic element, then for any compact neighborhood $K$ of $0$ in $\mathbb D$ there exist constants $\kappa_0>0$, $\kappa_1>1$, $\kappa_2>1$ such that 
for any $\alpha\in(0,\overline\alpha)$ and 
any $n\geq \max\{\kappa_{0}/\alpha,\min\left\{n\geq1\colon (1/n)\log n\leq\alpha/3\right\}\}$  we have
  \[\left|\mathscr{H}_n([\alpha,+\infty),K)\right| \leq 
\kappa_1\kappa_2^{I'(\alpha)}n^2e^{ -I(\alpha)n}.\]
   \end{itemize}
\end{theoremb}
This type of upper bounds were obtained for the Gauss map \cite{T20},
by extracting finite subsystems, applying to them the thermodynamic formalism \cite{Bow75,Rue04} and then using the regularity result of the Lyapunov spectrum \cite{KesStr07,PolWei99}. 
Our proof of Theorem~B is an extension 
of the argument in \cite{T20} to
 the Bowen-Series maps, based on the regularity result of the $\mathscr{H}$-spectrum in \cite{JaeTak22}.
 Unlike the Gauss map,
 the Markov shift in the proof of Theorem~A to which $f$ is semiconjugate is not the full shift. Hence, extracting finite subsystems is not straightforward. 
 If $G$ has a parabolic element,   %extracting finite subsystems with suitably bounded distortions
 it becomes more technical 
 since $f$ has a neutral periodic point and a distortion control is necessary.
 To this end,
  we use the induced Markov map constructed in \cite{JaeTak22}.
%The term $n^2$ in (b) originates in the error bound in the comparison between $t_n(\cdot)$ and $\log|(f^n)'|$ in Proposition~\ref{lem-logdist}.

%\begin{theorem}[Large deviations asymptotics on periodic points]\label{ldp2} For any $0<\alpha\leq\alpha^+$ and $n\geq1$, \[\sum_{\stackrel{x\in{\rm Per}_n(f)}{|(f^n)'x|\geq e^{\alpha n}}}|(f^n)'x|^{-\dim_H\Lambda(G)} \leq e^{  -\alpha(\dim_H\Lambda(G)-b(\alpha))n+o(n)}.\]  \end{theorem}

%\subsection{Methods of proofs and structure of the paper}
The rest of this paper consists of three sections. Section~2 provides preliminary results on the Fuchsian group actions and the Bowen-Series maps. We prove Theorem~A
in Section~3, and Theorem~B in Section~4.

\section{Preliminaries}
This section provides preliminary results on  Fuchsian groups  and  Bowen-Series maps. 
In Section~\ref{sec:cutting} we give formal definitions of  admissible fundamental domains and cutting sequences of oriented geodesics. 
  In Section~\ref{BowSer} we introduce the Bowen-Series map $f$.
 In Section~\ref{L=H} we prove key estimates comparing the homological growth rates and the growth rates of derivatives.
 After the definition of a Markov map
 in Section~\ref{markov-map},
we recall in Section~\ref{markov-sec} the construction of a finite Markov partition for $f$. 
 In Section~\ref{multi-section} we recall the results in \cite{JaeTak22}
 on the multifractal analysis of the homological growth rates.
%as far as we need them.

 % To show this,
 %we relate cutting orbits to BS orbits
 %(Lemma~\ref{parallel}),
 %and show that the growth rates along BS orbits are given by the decay rates of BS cylinders (Lemma~\ref{z-series}).
 %A main conclusion is the coincidence of the $\mathscr{H}$-spectrum and the Lyapunov spectrum of $f$ (Proposition~\ref{coincide-lem}).
 
%We use the letter $C$ to denote positive constants the value of which can differ from place to place.
 
\subsection{Cutting sequences for fundamental domains with even corners} \label{sec:cutting}
%{\sout{Let $R\subset \mathbb{D}$ be a  fundamental domain for $G$. By a fundamental domain we always mean a convex  and locally finite fundamental domain  which contains $0$ in its interior \cite{Bea83}. }} 
For $g\in G$, the inverse of $g$ is denoted by $\bar g$, and the 
word length of $g$ with respect to $G_R$ is denoted by $|g|$. Recall that $G_R$ is a  symmetric set of generators of $G$: $g\in G_R$ implies $\bar g\in G_R$. Since $G$ is of the first kind,
all the sides of $R$ are geodesics \cite[Theorem~12.2.12]{Rat94}. Since $G$ is finitely generated, $R$ has finitely many sides. 
%{\sout{The sides of $R$ give rise to a finite set $G_R$ of side-pairing transformations.}} 

 The copies of $R$ adjacent to $R$ along the sides of $R$ are of the form $eR$,
$e\in G_R$. For all $g\in G$  and $e\in G_R$, we label the side common to $gR$ and $geR$ on the side of $geR$ by $e$, and  on the side of $gR$ by $\bar e$. 
%\textcolor{red}{We label the side $s$ common to $R$ and $eR$ on the side of $eR$ by $e$, and that on the side of $R$ by $\bar e$.}
%\textcolor{red}{(remove) Equivalently, the side of $s$ interior to $R$ is labeled by the isometry which pairs it to some other side of $R$.} \textcolor{red}{(remove)We extend this labeling by translation under $G$ to the tessellation of $\mathbb D$:  corresponding sides of copies of $R$ have the same interior/exterior labels. }
%\bigcup_{g\in G}g\partial R$. 
By a side or vertex of %the net
$N=G\partial R$ %of images of $\partial R$ under $G$,
we mean the $G$-image of a side or vertex of $R$. 
We say   $R$  has {\it  even corners} 
\cite{BowSer79,Ser86}
 if  $N$ is a union of complete geodesics.  
% \textcolor{blue}{(INcorrect) We say $R$ is {\it admissible} if $R$ has even corners and  if it has at least four sides and satisfies the following property: If  $G$ has no parabolic element and $R$ has precisely four sides, then at least three geodesics in $N$ meet at each vertex of $R$. \cite[Theorem~3.1]{Ser86}.(awkward)}
%  \textcolor{cyan}{(suggstion)(Incorrect) We say $R$ is {\it admissible} if it has even corners, has at least four sides, and in case it has precisely four sides and $G$ has no parabolic element \cite[Theorem~3.1]{Ser86}, at least three geodesics in $N$ meet at each vertex of $R$.}
 We say $R$ is {\it admissible} if $R$ has even corners with  at least four sides and satisfies the following property: if  $R$ has precisely four sides with all vertices in $\mathbb{D}$ then at least three geodesics in $N$ meet at each vertex of $R$ \cite[Theorem~3.1]{Ser86}.
 The even corner assumption is not as restrictive as it appears: %In fact,
 every surface which is uniformized by a finitely generated Fuchsian group has a fundamental domain with this property  (see  \cite[Section 3]{BowSer79} and \cite[p.609, l.9-10]{Ser86}).

%  The collections $\widetilde F^s$ and $\widetilde F^u$ are the stable and unstable foliations of the first-return map.
%  Let $\widetilde F^s=\{\widetilde W^s(\gamma)\colon\gamma\in\mathscr{R}\}$ and $\widetilde F^u=\{\widetilde W^u(\gamma)\colon\gamma\in\mathscr{R}\}$.
  %namely the set $\Sigma\cap\{\gamma_0\in \mathbb S^1M\colon d(\phi_t(\gamma),\phi_t(\gamma_0))\to 0\quad (t\to+\infty)\}$
 % where $d$ denotes the Riemannian distance on the unit tangent bundle of $M$.
%  If $M$ is non-compact, this is not the case.

Unless otherwise stated we assume all geodesics are complete.
  If $\gamma\in\mathscr{R}$ %is an oriented geodesic which 
  passes through a vertex $v$ of $N$ in $\mathbb D$, we make the convention that $\gamma$ is replaced by a curve deformed to the right around $v$.
%From now on, 
We shall take as understood that all geodesics in $\mathscr{R}$ have been deformed, where necessary, in this way. %\footnote{Hence, it makes sense to say that a geodesic in $\mathscr{R}$ is a side of $N$.}.

%Let $\mathscr{R}$ denote the set of oriented geodesics joining two points in $\mathbb S^1$ and intersecting the interior of $R$. 
%  Let $\gamma\in\mathscr{R}$ and let $(e_{i_n})_{n=0}^{\infty}$ be the cutting sequence of $\gamma$.
%For $n\ge 1$ we define
%\[t_{n}(\gamma)=d(0,e_{i_0}e_{i_1}\cdots e_{i_{n-1}}(0)).
%\] 
%If $\gamma\in\mathscr{R}$
%is the oriented geodesic from  $\xi\in\mathbb S^1$
%to $\eta\in\mathbb S^1$,
% we call $\eta$ the positive endpoint of $\gamma$, and denote $\gamma^+=\eta$. 
   For  $\gamma \in \mathscr{R}$ we define an
  infinite 
  sequence 
  $(g_n)_{n=0}^\infty$ in $G_R$, called the {\it  cutting sequence} of $\gamma$ as follows:
  $g_0$ is the exterior label of the side of $R$ across which $\gamma$ crosses from $R$ to
  $g_0R$, 
  and for each $n\geq1$, $g_n$
  is the exterior label of the side of 
  $g_0\cdots g_{n-1}R$
  across which $\gamma$ crosses from 
  $g_0\cdots g_{n-1}R$ to $g_0\cdots g_{n}R$.

\begin{figure}
\begin{center}
\includegraphics[height=10cm,width=11cm]{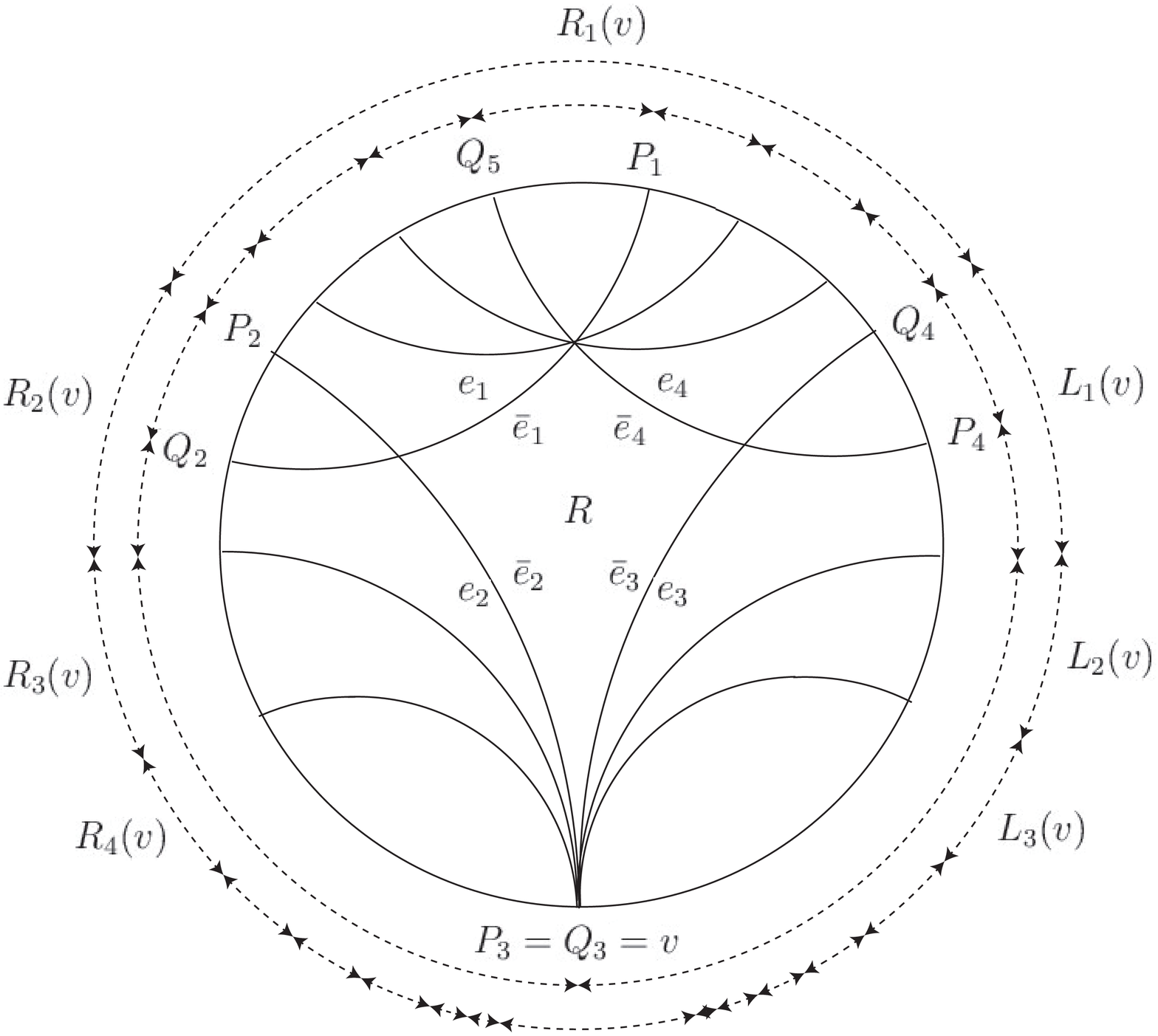}
\caption{A fundamental domain $R$ of a finitely generated, non-free Fuchsian group of the first kind with four sides:
 $e_1$ and $e_4$,  $e_2$ and $e_3$ are identified in pairs, 
which yields a non-compact hyperbolic surface of genus $0$ with one ramification point. The arcs in the inner dotted circle indicate elements of the finite Markov partition $(\varDelta(a))_{a\in S}$ constructed in Section~\ref{markov-sec}.}\label{fig:BS}
\end{center}
\end{figure}

\subsection{The Bowen-Series map}\label{BowSer}
 Let $m$ denote the number of sides of the fundamental domain $R$, with exterior labels
 $e_1,\ldots,e_m$ in anticlockwise order. 
For $1\leq i\leq m$ let $C(\bar e_i)$  denote the Euclidean closure of the geodesic which contains the side of $R$ with  exterior label $e_i$.
   We denote the two endpoints of $C(\bar e_i)$ by $P_i$ and $Q_{i+1}$ in anticlockwise order.
 For $j\in\mathbb Z$ with $i=j$ mod~$m$, set $e_j=e_i$, $P_j=P_i$, $Q_{j}=Q_{i}$. 
According to \cite{BowSer79,Ser86},
  the Bowen-Series map $f\colon\mathbb S^1\to \mathbb S^1$ is given by
\[f|_{[P_i,P_{i+1})}(x)=\bar e_ix.\]
For each $i\in\mathbb Z$, the restriction of $f$ to
$(P_i,P_{i+1})$
is analytic and can be extended to a $C^\infty$ map on $[P_i,P_{i+1}]$.
The derivatives of $f$ at points $P_i$ 
are the appropriate one-sided derivatives.
If $P_i$  is a cusp, then it is a neutral periodic point of $f$. 

\begin{remark}
Unlike \cite{BowSer79}, we do not assume that $C(\bar e_i)$ is contained in the isometric circle of $\bar e_i$. This means that the useful condition
 $\inf_{\mathbb S^1}|f'|\geq1$ may not hold.
\end{remark}

 The {\it  $f$-expansion} of a point $\xi\in\mathbb S^1$
is the one-sided infinite sequence $\xi_f=(e_{i_n})_{n=0}^{\infty} \in G_R^{\mathbb N}$ given by
$f^n(\xi)\in 
[P_{i_n},P_{i_{n}+1})$ for $n\ge 0.$
 %In what follows, 
 We will denote the $f$-expansion of $\xi\in\mathbb S^1$ by
 $(a_n)_{n=0}^\infty\in G_R^{\mathbb N}$
 to make a notational distinction from cutting sequences of geodesics. Note that $\overline{a_0\cdots a_{n-1}}\xi=f^n(\xi)$.

\subsection{Comparing homological growth rates and growth of derivatives}\label{L=H}
For the rest of the paper, 
we assume $G$ has an admissible fundamental domain $R$, and $f$ is the associated Bowen-Series map.

\begin{prop}\label{lem-logdist-nopar}
Assume that $G$ has no parabolic element. There
exists a constant $C_0>0$ 
such that for each geodesic $\gamma\in\mathscr{R}$ with  the infinite cutting sequence $(g_n)_{n=0}^{\infty}$  we have 
\[
\left|t_n(\gamma)-\log\left|(\overline{ g_{0}\cdots g_{n-1}})'\gamma^+\right|\right|\le C_0\ \text{ and }\ |t_n(\gamma)-\log|(f^n)'\gamma^+|| \le C_0.  \] 
\end{prop}

%If $G$ is a free group,   cutting sequences of geodesics directly give a symbolic coding of the limit set by a Markov shift \cite{Ser86}, and so the comparison can be made purely by  hyperbolic geometry. If $G$ is not a free group,  cutting sequences do not have a direct link to the dynamics on the limit set. We bypass this difficulty using the  results for Fuchsian groups with even corners  \cite{BowSer79, Ser86}. Lemma~\ref{lem-parallel}.

%{\sout{Let $\gamma\in\mathscr{R}$ have the infinite cutting sequence $(g_n)_{n=0}^\infty$, and let   $(a_n)_{n=0}^{\infty}$ denote the $f$-expansion of $\gamma^+$. We call $(g_0\cdots g_n0)_{n=0}^\infty$ the {\it  cutting orbit}, and  $(a_0\cdots a_n0)_{n=0}^\infty$ the {\it BS orbit} associated with $\gamma$.}} 

We introduce key ingredients for a proof of Proposition~\ref{lem-logdist-nopar}. The {\it Busemann function} is a function of 
 $\xi\in \mathbb{S}^1$ and $a,b\in \mathbb{D}$  given by \[B_{\xi}(a,b)=\lim_{t\rightarrow+\infty}\left(d(a,\gamma(t))-d(b,\gamma(t))\right),\] where  $\gamma=\{\gamma(t)\}_{t\in\mathbb R}$ is an oriented geodesic with $d(\gamma(s),\gamma(t))=|s-t|$ for $s,t\in\mathbb R$ whose positive endpoint is $\xi$. The limit always exists and is  independent of the choice of $\gamma$. 
%\textcolor{red}{Depends only on $a$, $b$, $\gamma^+$. I think. Yes. Would you like to add this? So, it depends on the choice of $\gamma$.
%Is the following what you intend? ``$B_{\xi}(a,b)=\lim_{t\rightarrow+\infty}\left(d(a,\gamma(t))-d(b,\gamma(t))\right)$, where $\gamma$ is a geodesic in $\mathscr{R}$ whose positive endpoint is $\xi$. This definition is independent of $\gamma$."} 
 The {\it Poisson kernel} is a function of
 $x\in \mathbb{D}$ and $\xi \in \mathbb{S}^1$ given by \[P(x,\xi)=\frac{1-|x|^2}{|\xi-x|^2}.\] It is well known that (see  \cite[Example~8.24]{bridsonhaefliger}) 
\begin{equation}\label{poisson}
   B_{\gamma^+}(0,b)=\log P(b,\gamma^+)=|\tau'(\gamma^+)|, 
\end{equation}
where $\tau$ is a M\"obius transformation  preserving $\mathbb{D}$ and satisfying $\tau(b)=0$. 

If $G$ is a free group, 
  the cutting sequence of $\gamma \in \mathscr{R}$  coincides with the $f$-expansion of $\gamma^+$, and so $(\overline{ g_{0}\cdots g_{n-1}})'\gamma^+=(f^n)'\gamma^+.$
  %{\sout{and hence the cutting and BS orbits associated with $\gamma$ coincide.}}
  If $G$ is not a free group,
 this is not always the case.
 The next lemma implies that the cutting sequence of $\gamma\in\mathscr{R}$ and the $f$-expansion of $\gamma^+$ differ only slightly.
%Nonetheless, they differ only slightly by the above lemma.

 \begin{lemma}[\cite{JaeTak22} Lemma~2.7]\label{lem-parallel}
Let $\gamma\in\mathscr{R}$ have the infinite cutting sequence 
  $(g_n)_{n=0}^\infty$ and let $(a_n)_{n=0}^\infty$ be the $f$-expansion of $\gamma^+$.  For any $n\geq0$,
%\textcolor{red}{(remove) $e_{i_0}\cdots e_{i_{n}}(R)$} 
 $g_0\cdots g_nR$ and 
$a_0\cdots a_{n}R$ share a common side of $N$, or else
share a common vertex of $N$ in $\mathbb D$.
\end{lemma}

\begin{proof}[Proof of Proposition~\ref{lem-logdist-nopar}]
    For $n \geq1$ we put
$x_{n}=g_0\cdots g_{n-1}0$.
Since $G$ has no parabolic element,  $R$ has a finite hyperbolic diameter, denoted by ${\rm diam}_{\rm h}(R)$.
Put $C_0=  4 {\rm diam}_{\rm h}(R)$.
It follows that   for all $n\in\mathbb{N}$, 
\begin{equation}\label{eq:passing-bdd-dist}
  \min\{d(x_n,y)\colon y\in\gamma\}\le \frac{C_0}{4}.
\end{equation}
%To prove this, first note that since $\gamma$ has an infinite cutting sequence,  the  endpoint $\gamma^+$ belongs to $\Lambda_c$. The  estimate in \eqref{eq:passing-bdd-dist} follows immediately if  $\gamma$ belongs to the Nielsen region of $G$. For the general case, we use that   cutting orbits of geodesics with the same endpoint in the limit set run in parallel in the sense of Lemma~\ref{lem-parallel}. Combining this with the fact that  geodesics with the same endpoint which emerge from a compact set $K$ have a bounded Hausdorff distance, the estimate in  \eqref{eq:passing-bdd-dist} follows.
%\textcolor{red}{\eqref{eq:passing-bdd-dist} is an immediate consequence of the fact that the hyperbolic diameter of $R$ is finite. Am I wrong?}

We denote by  $p_{n}$  the intersection of $\gamma$ and the horocircle
at $\gamma^+$ through $x_{n}$. Clearly, \eqref{eq:passing-bdd-dist} together with the triangle inequality implies
\begin{equation}
d(x_{n},p_{n})\le   \frac{C_0}{2}.\label{eq:uniform-estimate}
\end{equation} 

Fix $p_{0}\in\gamma\cap R$.  
%{\tiny
%Combining ? and ? (see also  \cite{KesStr04}) we obtain 
%\[
 % B_{\gamma^+}(0,x_{n})=\log P(x_n,\gamma^+)=\log\left|(\overline{ g_{0}\cdots g_{n-1}})'\gamma^+\right|.  
%\]
%Note that 
%\[
%B_{\gamma^+}%(p_{0},x_{n})=d(p_{0},p_{n})
%\]
%and
%\[
%\left|B_{\gamma^+}(p_{0},x_{n})-B_{\gamma^+}(0,x_{n})\right|\le d(0,p_{0})\le\text{diam}(K).
%\]
%By the triangle inequality we have 
%\[
%\left|d(p_{0},x_{n})-d(p_{0},p_{n})\right|\le d(x_{n},p_{n})
%\]
% and 
%\[
%\left|d(0,x_{n})-d(p_{0},x_{n})\right|\le d(0,p_{0})\le{\rm diam}(K).
%\]
%}
By combining  \eqref{eq:uniform-estimate} with the triangle inequality  and the fact that $B_{\gamma^+}(p_{0},x_{n})=d(p_{0},p_{n})$, %\textcolor{red}{(WHY: because $p_0$ is on the geodesic gamma, and $x_n$ projects onto $p_n$, +standard argument)},
we obtain 
\[\begin{split}
\left|d(0,x_{n})-B_{\gamma^+}(0,x_n)\right|
\le& |d(0,x_n)-d(p_0,x_n)|
+|d(p_0,x_n)-d(p_0,p_n)|\\
&+|B_{\gamma^+}(p_{0},x_{n})-B_{\gamma^+}(0,x_n)|\le  C_0.
\end{split}\]
%{\tiny
%By combining  \eqref{eq:uniform-estimate} with the triangle inequality and the fact that $B_{\gamma^+}(p_{0},x_{n})=d(p_{0},p_{n})$, we obtain 
%\[\begin{split} 
%\left|d(0,x_{n})-\log\left|(\overline{ g_{0}\cdots g_{n-1}})'\gamma^+\right|\right|
%\le& |d(0,x_n)-d(p_0,x_n)|
%+|d(p_0,x_n)-d(p_0,p_n)|\\
%&+|B_{\gamma^+}(p_{0},x_{n})-B_{\gamma^+}(0,x_n)|\\
%\le& C_1+2\cdot{\rm diam}(K).
%\end{split}\]
%}
 The first assertion follows from this and  \eqref{poisson} with  $b=x_n$.
The second one follows from the same argument, replacing $g_0\cdots g_{n-1}$
by $a_0\cdots a_{n-1}$ and using the uniform bound on $d(x_n,a_0\cdots a_{n-1}0)$ guaranteed by Lemma~\ref{lem-parallel}. 
\end{proof}

\begin{prop}\label{lem-logdist}
Assume that $G$ has a parabolic element. Let $K$ be a compact neighborhood of $0$ in $\mathbb{D}$. There
exists a constant $C_0>0$ %\textcolor{red}{(remove) depending  only on $K$,  $R$ and the Fuchsian group $G$} 
such that for each geodesic $\gamma\in\mathscr{R}$ which intersects
$ K \cap R $ and has   an infinite cutting sequence $(g_n)_{n=0}^{\infty}$  we have 
\[
\left|t_n(\gamma)-\log\left|(\overline{ g_{0}\cdots g_{n-1}})'\gamma^+\right|\right|\le 2\log n+C_0, \text{ and}
\]
 \[ |t_n(\gamma)-\log|(f^n)'\gamma^+|| \le 2 \log n +C_0.  \] 

%\textcolor{blue}{Slight improvement: We can choose $C$ independent of $K$ if we restrict to $n$ large enough. Here, large enough means $n\ge N(\gamma)$ the first return time of $\gamma$ to some fixed compact set depending only on $G$.}
%\[
%\left|t_n(\gamma)-\log\left|(\tau_{n}(\gamma)^{-1})'\right|(\gamma^+)\right|\le C\log(n).
%\]
\end{prop}

\begin{proof}
%If $(e_i)_{i=0}^\infty$ is the cutting sequence of $\gamma$, then  $\gamma$ successively crosses the fundamental domains
%\[
%e_0\cdots e_{n-1}R,\quad %n\ge 1
%\]
%Then $\tau(\gamma)=\bar e_0 \circ \gamma$. 
%{\tiny
%\textcolor{red}{(Under construction)} 
%We will show the existence of 
%$C_1>0$ depending on $K$,  $G$ and $R$ such that for every $\gamma$ crossing
%$R\cap K$ and $n\in\mathbb{N}$, 
%\begin{equation}
%d(x_{n},p_{n})\le 2\log n+C_1.\label{eq:log-estimate}
%\end{equation}
%Once (\ref{eq:log-estimate}) has been  established, we can proceed  as in the proof of Lemma ?? with (\ref{eq:uniform-estimate}) replaced by (\ref{eq:log-estimate}). We omit the details. 
%We now turn to the proof of \eqref{eq:log-estimate}. 
%}
By passing to an iterate of $f$, we may assume that all cusps of $R$ are fixed points of $f$. We  follow  the argument in  \cite[p.161~(3)]{KesStr04}. 
Let $r_0>1$ be  a large number to be determined later. For each cusp of $R$, we conjugate the cusp to infinity in the upper half-plane model and $0$ to $i$. For each cusp, we remove from $R$ the horodisk $H(r_0)=\{z\in\mathbb C \colon {\rm Im}(z)>r_0 \}$.  This defines a compact subset $K(r_0)$ of $R$. We take  $r_0$ large enough so that $K\subset K(r_0)$ and the horodisks associated with the cusps are pairwise disjoint.

For $n \geq1$ we put
$x_{n}=g_0\cdots g_{n-1}0$, and denote by  $p_{n}$  the intersection of $\gamma$ and the horocircle
at $\gamma^+$ through $x_{n}$. First assume that  $\gamma \cap g_0\cdots g_{n-1}(R \cap K(r_0))$ is non-empty. Then $\gamma$ passes within the distance ${\rm diam_{\rm h}}(K(r_0))$  of $x_n$, which  implies 
\begin{equation}\label{uniform-estimate1}
    d(x_n,p_n)\le  2 {\rm diam_{\rm h}}(K(r_0)).
\end{equation}

Now assume that $\gamma \cap g_0\cdots g_{n-1}(R \cap K(r_0))$ is empty. Then $\gamma \cap g_0\cdots g_{n-1}R $ is contained in one of the horodisks that are the $g_0\cdots g_{n-1}$-images of the horodisks removed from $R$. 
Conjugating the $g_0\cdots g_{n-1}$-image of the cusp in this horodisk  to infinity in the upper half-plane model and $x_n$ to $i$ we see that $\gamma \cap g_0\cdots g_{n-1}R\subset H(r_0)$. 
Let $m\le n$ denote the smallest integer such that for all $k\in \{m,m+1,\ldots,n\}$ we have  $\gamma\cap g_0\cdots g_{k-1}R  \subset H(r_0)$. Since the cusp is a fixed point of $f$, we have  $x_k = i+ (n-k)\lambda$  for all $k\in \{m,m+1,\ldots,n\}$,  with some 
$\lambda\in\mathbb R\setminus\{0\}$ depending only on the cusp. 
Let $x'$ denote the unique point on $\gamma$  which satisfies ${\rm Re}(x')<0$ and ${\rm Im}(x')=1$ (see \textsc{Figure}~\ref{fig:triangle}). There exists a uniform constant $C'=C'(r_0,\lambda)>0$ such that $d(x_m,x')\le C'$. 
Hence, there exists a constant $C''=C''(r_0,\lambda)>0$ such that $|x_m-x'|\le C''$. 

\begin{figure}
\begin{center}
\includegraphics[height=5cm,width=10cm]{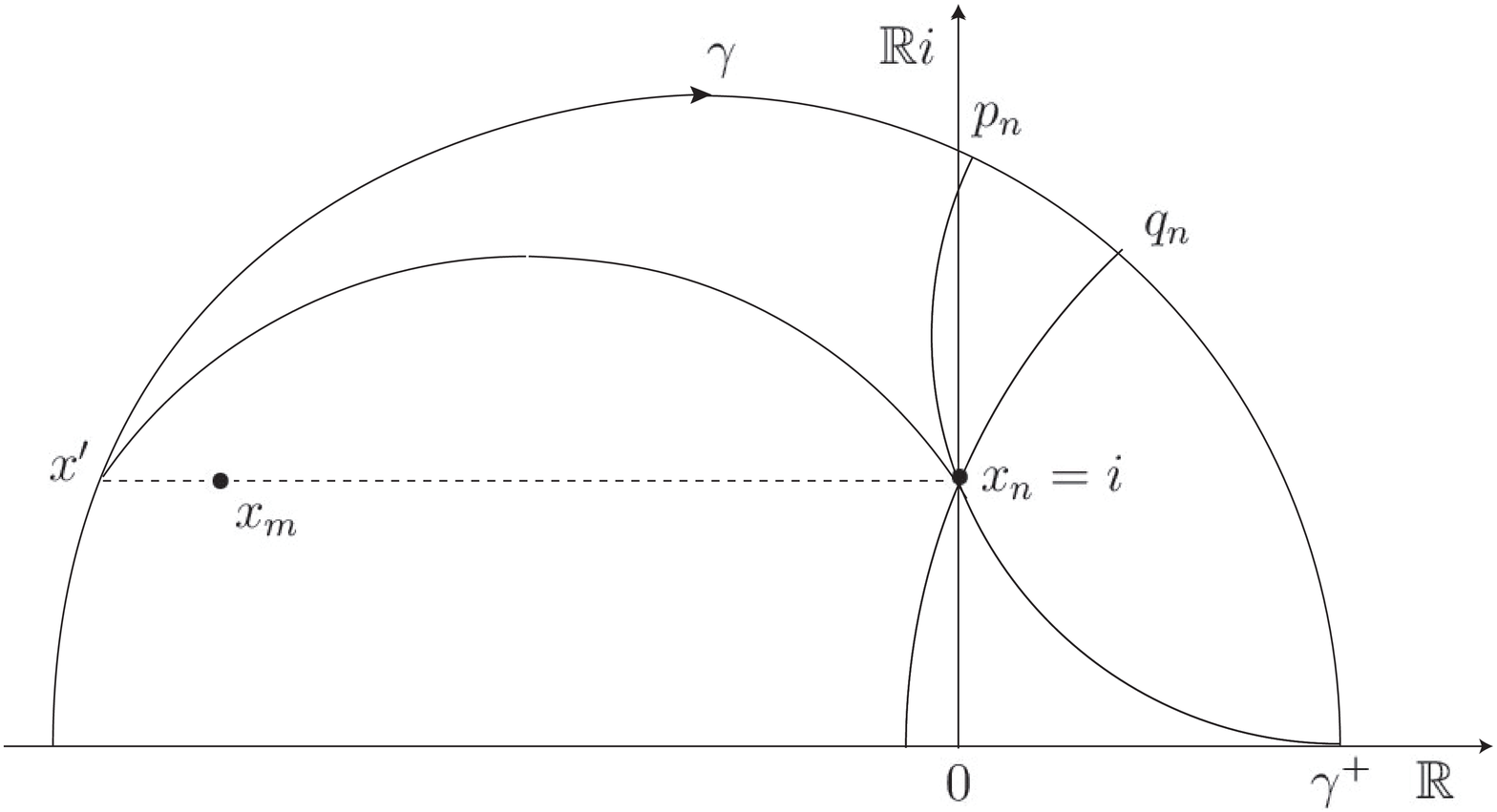}
\caption{On the proof of Proposition~\ref{lem-logdist} in the case $\gamma \cap g_0\cdots g_{n-1}(R \cap K(r_0))=\emptyset$  and $\lambda>0$.}
\label{fig:triangle}\end{center}
\end{figure}

We verify the existence of a constant 
%\textcolor{red}{(remove) As in \cite[Proposition~2.2]{JKM21} one verifies that there exists a uniform constant }
$C=C(r_0,\lambda)>0$ such that
\begin{equation}\label{log-nm}
    d(x',x_n)\le  2\log(n-m) +C.
\end{equation}
%\[d(x_m, x_n)\le  2\log(n-m) +C.\]
 Let $r>1$ denote the Euclidean radius of the geodesic arc between $x'$ and $x_n$. Then  \[d(x',x_n)=2\log(\sqrt{r^2-1}+r)\] by \cite{floyd80}. Hence, for $r$ sufficiently large, $d(x',x_n)\leq 2\log(2\sqrt{r^2-1}+1)$. By the  Euclidean Pythagorean theorem we have 
\[2\sqrt{r^2-1}=|x'-x_n|\leq |x'-x_m|+|x_m-x_n|\leq|\lambda|(n-m)+C''.\] 
Taking $C=2\log|\lambda|+2\log\left(1+(C''+1)/|\lambda| \right)$ we obtain
\[d(x',x_n)\leq 2\log(|\lambda|(n-m)+C''+1)\le 2\log(n-m) + C.\] This proves \eqref{log-nm} for all $r$ sufficiently large.   Enlarging $C$ if necessary, we can show \eqref{log-nm} for all $r>1$.  The proof of \eqref{log-nm} is complete.

Let $q_n$ denote the point of intersection between $\gamma$ and the geodesic through $x_n$ that is orthogonal to $\gamma$ (see \textsc{Figure}~\ref{fig:triangle}). Since the  hyperbolic triangle with vertices  $x'$, $x_n$, $q_n$ has a right angle at $q_n$, the 
hyperbolic Pythagorean theorem implies $d(x_n,q_n)< d(x',x_n) + 2\log2$.
By \eqref{log-nm}   it  follows that
\begin{equation} \label{nonuniform-estimate}
\begin{split}
d(x_n,p_n)< d(x_n,q_n)&<d(x',x_n) + 2\log 2\\&\le 2\log(n-m)+C +2\log 2\\
&\leq 2\log n+C+2\log2.
\end{split}
\end{equation}

Proceeding exactly as in the proof of Proposition~\ref{lem-logdist-nopar} replacing \eqref{eq:uniform-estimate} by  \eqref{uniform-estimate1} and \eqref{nonuniform-estimate}, we obtain the first desired estimate in Proposition~\ref{lem-logdist}.
For the second one, we proceed in the same way replacing $g_0\cdots g_{n-1}$
by $a_0\cdots a_{n-1}$ and using $g_k=a_k$ for $m\leq k\leq n$ and Lemma~\ref{lem-parallel}. 
%{\tiny
%$d(x_m, x_n)\le 2\log(n-m) +C \le 2\log(n) +C$, where $C$ depends on $\lambda$. This gives 
%\[
%B_{\gamma^+}(x_m,x_n)\le d(x_m,x_n) \le 2\log(n) +C.
%\]
%Therefore, \textcolor{red}{NOT OK}
%\[
%d(0,x_n)-B_{\gamma^+}(0,x_n) \le 
%|d(0,x_n)-B_{\gamma^+}(0,x_m)|+2\log(n)+C \le 2\log(n)+C'.
%\]
%The first assertion follows from this and \eqref{poisson} with $C=C'$. For the second, we use Lemma~\ref{lem-parallel} as in Lemma ?. } 
%The statement of the lemma then reads
%\[
%\left|t_n(\gamma)-\log\left|(\bar e_{n-1}\cdots\bar %e_0)'\right|(\gamma^+)\right|\le C\log(n).
%\]
 %one verifies that there exists
%$C>0$ depending on $K$ and $G$ and $R$ such that for every $\gamma$ crossing
%$R\cap K$ and $n\in\mathbb{N}$, %$d(x_n,\gamma\cap e_0 \cdots e_{n-1}R)\le 2\log(n)+C$. This estimate implies, with  $p_{n}$ denoting the intersection of $\gamma$ and the horocircle
%at $\gamma^+$ through $x_{n}$, 
\end{proof}

\subsection{Markov maps}\label{markov-map}
Let $S$ be a discrete %countable 
set with $\#S\ge 2$.
 Given 
a set $\Sigma$ of one-sided infinite sequences
 $(x_n)_{n=0}^\infty=x_{0}x_{1}\cdots$ in the cartesian product topological space $S^{\mathbb N}$, let
$E(\Sigma)$ denote the set of finite words in $S$ that appear in some element of $\Sigma$. For an integer $n\geq1$, let $E^n(\Sigma)$
denote the set of elements of $E(\Sigma)$ with word length $n$.

   A {\it  Markov map} is a map $F\colon \Gamma\to \mathbb S^1$ such that the following holds:
    \begin{itemize}
 \item[(i)]
   There exists a family $(\Gamma(a))_{a\in S}$ 
of pairwise disjoint arcs in $\mathbb S^1$ 
such that $\Gamma=\bigcup_{a\in S}\Gamma(a)$.

 \item[(ii)] For each $a\in S$, the restriction 
 $F|_{\Gamma(a)}$
extends to a $C^1$ diffeomorphism from  
 the closure of $\Gamma(a)$ onto its image. 
 \item[(iii)] If $a,b\in S$ and $F(\Gamma(a))\cap \Gamma(b)$ 
 has non-empty interior,
 then $F(\Gamma(a))\supset \Gamma(b)$.
 
 \end{itemize}
 The family  $(\Gamma(a))_{a\in S}$ of arcs is called a {\it  Markov partition} of $F$.

 Condition (iii) determines a transition matrix $(M_{ab})$ over the countable alphabet $S$ given by  $M_{ab}=1$ if $F(\Gamma(a))\supset \Gamma(b)$ and $M_{ab}=0$ otherwise.
It determines a countable topological Markov shift
 $\Sigma=\Sigma(F,(\Gamma(a))_{a\in S})$ by
\begin{equation}\label{m-shift}\Sigma=\{x=(x_n)_{n=0}^{\infty}\in S^{\mathbb N}\colon  M_{x_nx_{n+1}}=1 \ 
\text{for } n \ge 0\}.\end{equation}
The relative topology on $\Sigma$ has a base that consists of sets of the form
\[[\omega_0\cdots\omega_{n-1}]=\{x\in   \Sigma\colon x_k=\omega_k\ \text{for }0\leq k\leq n-1 \},\ n\ge 1, \ \omega_0\cdots \omega_{n-1}\in S^n.\]
%called {\it cylinders}.
This topology is metrizable
with the metric $d_\Sigma$ given by $d_{\Sigma}(x,y)=
\exp(-\min\{n\geq0\colon x_n\neq y_n\})$
for distinct points $x,y\in\Sigma$.

% A word  $\omega_0\cdots \omega_{n-1}\in Q^{n}$  is {\it admissible} if $n=1$, or else $n\geq2$ and $T_{\omega_{j}\omega_{j+1}}=1$ holds for all $0\leq j\leq n-1$.

 If  $\bigcap_{n=0}^{\infty}\overline{ F^{-n}(\Gamma(x_n))}$ is a singleton for all $(x_n)_{n=0}^\infty\in \Sigma$, 
 then we define a coding map $\pi_{\Sigma}\colon \Sigma\to \mathbb S^1$ by 
  \begin{equation}\label{code-map}\pi_{\Sigma}((x_n)_{n=0}^{\infty})\in 
 \bigcap_{n=0}^{\infty}\overline{ F^{-n}(\Gamma(x_n))}.\end{equation}
 The coding map is continuous and semiconjugates $F$ to the left shift on $\Sigma$.

For $\omega\in S^m$ and $\zeta\in S^n$,
 write $\omega\zeta$ for the concatenation
$\omega_0\cdots \omega_{m-1}\zeta_0\cdots \zeta_{n-1}\in S^{m+n}$. 
 For convenience, put $E^0=\{ \emptyset\}$,  $|\emptyset|=0$, and  $\omega\emptyset=\omega=\emptyset\omega$ 
 for  all  $\omega\in E(\Sigma)$.

 \subsection{Construction of a finite Markov partition}\label{markov-sec}
A point $v\in\mathbb S^1$ is called a {\it  cusp} of $R$ if $v$ is the common endpoint of two sides of $R$.
The set of all cusps of $R$ is denoted by $V_c$. 
%Note that each $v\in V_c$ 
Each cusp is a fixed point of some parabolic element of $G$. %Conversely, 
If $G$ has a parabolic element, then $V_c$ is non-empty. 
 %If $G$ has a parabolic element, 
 %there is a cusp.
 
 Let $V$ denote the set of all vertices of $R$ in $\mathbb{D}\cup \mathbb{S}^1$.   %\textcolor{blue}{(add) and vertices of $R$ in $\mathbb{S}^1$}\textcolor{red}{(remove) , proper vertices and improper vertices of $R$.}
%\textcolor{red}{(remove?) Each vertex of $R$ that is not  improper is contained in $C(\bar e_{i-1})\cap C(\bar e_{i})$  for some  $i\in\mathbb Z$. We denote this vertex by $v_{i}$.} 
For each $v\in V$ we denote by $W(v)\subset\mathbb S^1$ the set of points where the geodesics in $N$ passing through $v$ meet $\mathbb{S}^1$. 
%the geodesics in $N(v)$ geodesics in \textcolor{red}{the closure of?} $N$ which intersect $v$, and by $W(v)$ the union of the endpoints of the geodesics in $N(v)$
%and denote by $W(v)$ the union of the endpoints of the geodesics in $N(v)$.
%\textcolor{red}{(remove)For each $i\in\mathbb Z$ we denote the arcs of $\mathbb S^1$ cut-off by successive points of $W(v_i)$ in anticlockwise order from  $P_i$ to $Q_{i+1}$ by $L_{k(v)}(v),\ldots, L_1(v)$, and the arcs in clockwise order from $P_i$ to $Q_{i+1}$ by $R_{k(v)}(v),\ldots, R_1(v)$. If $v$ is a cusp, we label the arcs clockwise from $Q_{i+1}$ to $Q_{i}=v$ as  $L_1(v),L_2(v),\ldots$, and anticlockwise from   $Q_{i+1}$ to $Q_{i}$ by $R_1(v),R_2(v),\ldots$, as in Figure~4. Note that for each vertex $v\in V$ that is not improper, \[\bigcup_{r=1}^{k(v)} L_r(v)\cup\bigcup_{r=1}^{k(v)} R_r(v)=\begin{cases} \mathbb S^1&\text{if $v\notin V_c$;}\\ \mathbb S^1\setminus\{v\}&\text{if $v\in V_c$}. \end{cases}\]}
The set $\bigcup_{v\in V}W(v)$ is $f$-invariant \cite[Lemma~2.3]{BowSer79}
and hence induces a Markov partition
for $f$. This partition is an infinite partition 
if and only if $R$ has a cusp.
If $R$ has a cusp,
below we construct a coarser finite Markov partition for $f$ that determines a transitive finite Markov shift.

Let $v\in V_c$. There exists  $i\in\mathbb Z$ such that $v\in C(\bar e_{i-1})\cap C(\bar e_{i})$.  We denote the arcs of $\mathbb S^1$ cut-off by successive points
of $W(v)$ in clockwise order from   $Q_{i+1}$ to $Q_{i}=v$ as  $L_1(v),L_2(v),\ldots$,
and anticlockwise from  %$[Q_{i+1},P_{i-1})$ as 
$Q_{i+1}$ to $Q_{i}$ by
$R_1(v),R_2(v),\ldots$, as in \textsc{Figure}~\ref{fig:BS}.  
%For each $v\in V_c$ 
We define
\[L(v)=\overline{ \bigcup_{r\geq 2}L_r(v)}\quad \text{and}\quad R(v)=\overline{\bigcup_{r\geq 3}R_r(v)}.\]

 For each $v\in V$ we define 
\[W'(v)=\begin{cases}W(v)&\text{if $v\notin V_c$,}\\
\partial L_1(v)\cup\partial L(v)
\cup\partial R_2(v)\cup\partial R(v)&\text{if $v\in V_c$,}\end{cases}\]
and put
\[W'=\bigcup_{v\in V}W'(v).\]
Note that $W'$ is a finite subset of $\bigcup_{v\in V}W(v)$. One verifies 
$f(W')\subset W'$ \cite[Lemma~3.1]{JaeTak22}.
%as in \cite[Lemma~2.3]{BowSer79}.
%the following lemma. 
%\begin{lemma}\label{endpoints-Markov}
%We have $f(W')\subset W'$. 
%\end{lemma}
 We define a partition of $\mathbb S^1$ into arcs 
 with endpoints given by two consecutive points in $W'$.
We choose all partition elements to be left-closed and right-open, in anticlockwise order. 
We label the partition elements by integers from  a finite subset $S$ of $\mathbb N$, and
denote the partition element labeled with $a\in S$ by $\varDelta(a)$.

%\textcolor{red}{(remove, moved prior to Prop 3.2) We define a partition of $\varDelta$ into arcs   with endpoints given by two consecutive points in $W'$. We choose all partition elements to be left-closed and right-open, in anticlockwise order. We label the partition elements by integers in a finite subset $S$ of $\mathbb N$, and denote the element labeled with $a\in S$ by $\varDelta(a)$. By the transitivity of $f|_{\Lambda}$ \cite[Lemma~ 2.5]{BowSer79} and Lemma~\ref{endpoints-Markov}, $f$ is a finitely irreducible Markov map with a finite Markov partition $(\varDelta(a))_{a\in S}$,  which } 

Then $f\colon\mathbb S^1\to  \mathbb{S}^1$ is a Markov map with a finite Markov partition 
$(\varDelta(a))_{a\in S}$ (see \cite[Proposition 3.2]{JaeTak22}), and
determines by \eqref{m-shift} a transitive finite Markov shift 
\[X=X(f,(\varDelta(a))_{a\in S}).\]
  By \cite[Lemma~2.5]{JaeTak22}, 
 the coding map 
 %\textcolor{red}{$\pi=\pi_X$ given by \eqref{code-map} is well defined and continuous.}
$\pi_X$ is well defined by \eqref{code-map}.
  Let $\sigma\colon X\rightarrow X$ denote the left shift given by 
$(\sigma x)_n=x_{n+1}$ for $n \ge 0$.
 We have 
 \[f\circ \pi_X=\pi_X\circ\sigma.\]

\subsection{Multifractal analysis of homological growth rates}\label{multi-section}
Let $Y$ be a topological space and let
 $F\colon Y\to Y$ be a Borel map.
 Let $\mathcal M(Y)$ denote the space of $F$-invariant Borel probbility measures endowed with the weak* topology, and 
let $\mathcal M(Y,F)$ denote
the set of elements of $\mathcal M(Y)$   which are invariant under $F$.
 For each $\mu\in\mathcal M(Y,F)$,
 let $h_\mu(F)$ denote the measure-theoretic entropy of $\mu$ with respect to $F$. If the dependence on $F$ is clear, we often write $h(\mu)$ for $h_\mu(F)$.
 For $\mu\in \mathcal M(\mathbb S^1,f)$, we define the {\it  Lyapunov exponent} of $\mu$ 
 by $\chi(\mu)=\int\log|f'| d\mu.$

Imitating the style of the Poincar\'e exponent, 
for each $\beta\in\mathbb R$ we define
%following  \cite[Theorem~2.1.3]{MauUrb03} \textcolor{red}{\cite{MauUrb03} to be removed, since we do not use thermodynamics for countable Markov shift}
%we introduce a {\it  generalized Poincar\'e exponent} at an inverse temperature $\beta \in \mathbb{R}$ by
\[P(\beta)=\inf \left\{t \in \mathbb{R} \colon \sum_{g\in G} \exp(-\beta d(0,g0)-t|g|)<+\infty\right\},\]
and call 
the function $\beta\in\mathbb R\mapsto P(\beta)$ 
%the {\it  geometric pressure function of $G$ with respect to $R$}, or simply 
a {\it pressure function}.
 Define 
\[
\beta_+=\sup\{\beta\in\mathbb R\colon P(\beta)>-\underline\alpha\beta\}.
\]
%\textcolor{red}{We should have added `non-increasing' in \cite[Main Theorem]{JaeTak22}.} \textcolor{blue}{Should we add "strictly decreasing" on $(-\infty, \beta_+)$? We showed in the previous paper that $P'<0$ on $(-\infty, \beta_+)$.  Also, should we replace "strict convex" by $P''>0$? }

\begin{figure}
\begin{center}
\includegraphics[height=9cm,width=12cm]{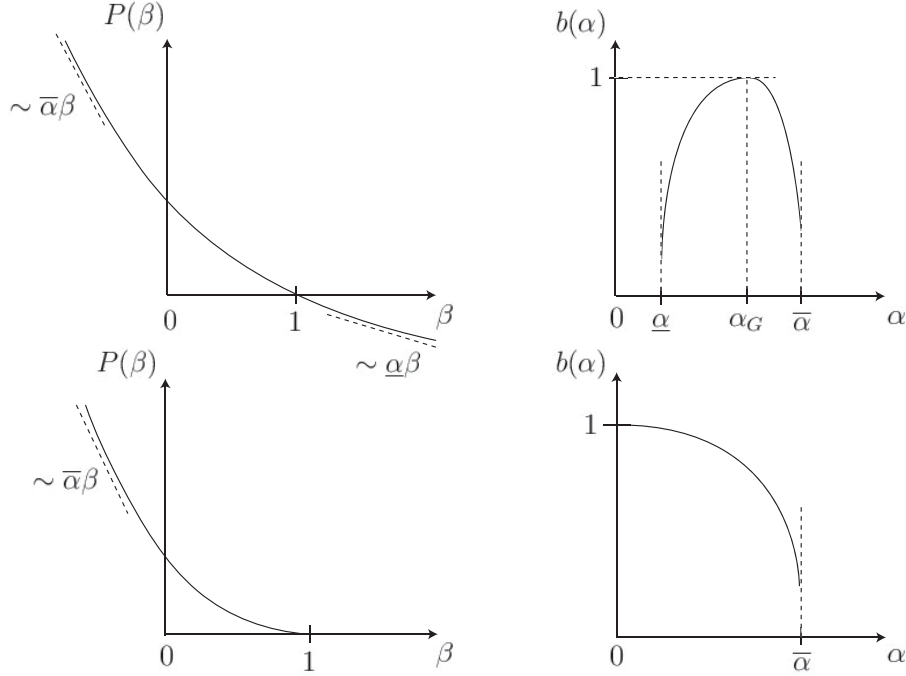}\end{center}
\caption{The graphs of the pressure function $\beta\in\mathbb R\mapsto P(\beta)$ and the $\mathscr{H}$-spectrum $\alpha\in[\underline\alpha,\overline\alpha]\mapsto b(\alpha)$: $G$ has no parabolic element (upper); $G$ has a parabolic element (lower).
%We have $\min\{\beta\geq0\colon P(\beta)=0\}=1$. 
%The constant $\alpha_G$ is the unique maximum point of the $\mathscr{H}$-spectrum.
}\label{spectrum}\end{figure}

\begin{theorem}
\label{bslyapunov}
Let $G$ be a finitely generated non-elementary Fuchsian group of the first kind with an admissible fundamental domain $R$ having even corners. 
\begin{itemize}
\item[(a)] We have
     $\underline{\alpha}=\alpha_-$ and $\overline{\alpha}=\alpha_+$. 
    \item[(b)] 
  The pressure function $P$
is convex, non-increasing and continuously differentiable on $\mathbb R$, and  analytic and strictly convex on $(-\infty,\beta_+)$.
If $G$ has no parabolic element, then
 $\beta_+=+\infty$. If  $G$ has a parabolic element, then $\beta_+=1$ and $P$ vanishes on $[1,+\infty).$ 
\item[(c)] We  have $\underline\alpha<\overline\alpha$, and  the level set $\mathscr{H}(\alpha)$ is non-empty if and only if $\alpha\in[\underline\alpha,\overline\alpha]$.   The $\mathscr{H}$-spectrum is
 %$\alpha\mapsto b(\alpha)$ 
continuous on $[\underline\alpha,\overline\alpha]$,  analytic on $(\underline\alpha,\overline\alpha)$, and for each $\alpha\in[\underline\alpha,\overline\alpha]\setminus\{0\}$ we have
\begin{equation}\label{l-formula}b(\alpha)=\max\left\{\frac{h(\mu)}{\chi(\mu)}
\colon\mu\in\mathcal M(\mathbb S^1,f),\ \chi(\mu)=\alpha\right\}.\end{equation}
 Moreover, the $\mathscr{H}$-spectrum attains its maximum $1$ at $\alpha_G\in [\underline\alpha,\overline\alpha)$,  is strictly increasing on  $[\underline\alpha,\alpha_G]$ and strictly decreasing on  $[\alpha_G,\overline\alpha]$, 
where \begin{equation}\label{ag}\alpha_G=-P'(1),\end{equation}
and $\lim_{\alpha\nearrow\overline\alpha}b'(\alpha)
=-\infty$.
  If $G$ has no parabolic element, then  $\alpha_G>\underline\alpha>0$ and $\lim_{\alpha\searrow\underline\alpha}b'(\alpha)=+\infty$.     If $G$ has a parabolic element, then  $\alpha_G=\underline\alpha=0$.
 \end{itemize}
     \end{theorem}
 Schematic graphs of the pressure and the $\mathscr{H}$-spectrum are shown in \textsc{Figure}~\ref{spectrum}.

To proceed, we need 
the following distortion property.
Let $\Sigma^+$ denote the set of $f$-expansions of points in $\mathbb S^1$.
Let $n\geq1$ and let $a_{0}\cdots a_{n-1}\in E^n(\Sigma^+)$. We define
%{\sout {a {\it  BS cylinder}, or more precisely  a BS $n$-cylinder by}}
\[\Theta(a_{0}\cdots a_{n-1})=\left\{\xi\in \mathbb S^1\colon
f^k(\xi)\in 
[P_{i_{k}},P_{i_{k}+1})\quad\text{for } 0\leq k\leq n-1  \right\},\]
where $a_k=e_{i_k}\in G_R$ for $0\leq k\leq n-1$.
Define
\[D_{n}=\sup_{a_0\cdots a_{n-1}\in E^n(\Sigma^+)}\sup_{x,y\in \Theta(a_0\cdots a_{n-1} )}\frac{|(f^{n})'x|}{|(f^{n})'y|}.\]
 \begin{prop}[\cite{JaeTak22} Proposition~2.8]\label{MILD}
We have $\log D_{n}=o(n)$ $(n\to\infty)$.
\end{prop}
\begin{remark}\label{D_n-remark} 
If $G$ has no parabolic element, some power of $f$ is uniformly expanding \cite[Theorem~5.1]{Ser81b}, and so $D_{n}$ is uniformly bounded. If $G$ has a parabolic element, 
Lemma~\ref{lem-parallel} and Proposition~\ref{lem-logdist} 
imply $D_n=O(n^4)$. %\textcolor{red}{(J) Is this really OK? Lemma 2.3 is valid only for limit points, and not for the whole BS-cylinders involved  in $D_n$. (H) OK, since limit points are dense (1st kind) and $(f^n)'$ is continuous on the BS $n$-cylinder.   } 
If moreover $\inf_{\mathbb S^1}|f'|\geq1$, then using the explicit form of parabolic elements in \cite[Lemma~2.7]{BowSer79} one can show that $D_n=O(n^2)$.
\end{remark}

\begin{proof}[Proof of Theorem~\ref{bslyapunov}]
Below we only prove (a) using the finite Markov partition constructed in Section~\ref{markov-sec}.
The rest of the assertions %is a summary of the results in 
are contained in \cite[Main Theorem, Proposition~5.3]{JaeTak22}.

 Since cutting sequences are shortest \cite[Theorem~3.1(ii)]{Ser86}, we have $\underline{\alpha}\leq\alpha_-$. 
To show the reverse inequality, take a sequence $(g_n)_{n=1}^\infty$ in $G\setminus\{1\}$
 such that $k$, $n\geq1$, $k\neq n$ implies $g_k\neq g_n$, and
 for all $n\geq1$,
 \begin{equation}\label{display1}0\leq\frac{d(0,g_n0)}{|g_n|}-\underline\alpha<\frac{1}{n}.\end{equation}
 Since $G$ is finitely generated,
we have $|g_n|\to\infty$ as $n\to\infty$.  Let $a_{n,0}\cdots a_{n,|g_n|-1}\in E^{|g_n|}(\Sigma^+)$ denote the admissible shortest representation of $g_n$ (see \cite[Lemma~3.9]{JaeTak22}). By the mean value theorem, there exists $C\geq1$ such that for all $\xi\in \Theta(a_{n,0}\cdots a_{n,|g_n|-1})$ we have
  \begin{equation}\label{display2}
  \frac{1}{CD_n}\leq\frac{|\Theta(a_{n,0}\cdots a_{n,|g_n|-1})|}{|(f^{|g_n|})'\xi|^{-1}}\leq CD_n.\end{equation}
   Pick $\omega_0\cdots\omega_{|g_n|-1}\in E^{|g_n|}(X)$ such that  $\varDelta(\omega_0\cdots\omega_{|g_n|-1})\subset\Theta(a_{n,0}\cdots a_{n,|g_n|-1})$.
 Since $X$ is a transitive finite Markov shift, there is an integer $L\geq1$ such that for each $n\geq1$
 there is a periodic point $\xi_n\in\varDelta(\omega_0\cdots\omega_{|g_n|-1})$ of period in $\{|g_n|,\ldots,|g_n|+L\}$.
 Let $\mu_n$ denote the uniform probability distribution
 on the $f$-orbit of $\xi_n$.
 
From the proof of \cite[Proposition~2.9]{JaeTak22}, there exists $C'\geq1$ such that 
 %for all $n\geq1$ and $a_0\cdots a_{n-1}\in E^n(\Sigma^+)$,
\begin{equation}\label{z-series-eq}\frac{1}{C'}%\alpha_{\rm min}
\leq\frac{|\Theta(a_{n,0}\cdots a_{n,|g_n|-1})|}{\exp(-d(0,a_{n,0}\cdots a_{n,|g_n|-1}0))}\leq C'
D_{|g_n|}.\end{equation}
 Combining \eqref{display1} and \eqref{display2} with \eqref{z-series-eq} and using Proposition~\ref{MILD} 
 %the fact that $f$ has mild distortion 
 implies  $\underline\alpha\geq\limsup_n\chi(\mu_n)$.
 Since $\alpha_-=\inf\{\chi(\mu)\colon\mu\in\mathcal M(\mathbb S^1,f)\}$ by \cite[Lemma~3.6]{JaeTak22}, we have $\limsup_n\chi(\mu_n)\geq\alpha_-$.
We have verified 
 $\underline{\alpha}=\alpha_-$. The  proof of $\overline{\alpha}=\alpha_+$ is analogous, with
 $\alpha_+=\sup\{\chi(\mu)\colon\mu\in\mathcal M(\mathbb S^1,f)\}$ also by \cite[Lemma~3.6]{JaeTak22}.
 \end{proof}

\section{Large deviations}
This section is dedicated to the proof of Theorem~A.
 In Section~\ref{ld-shift} we prove the level-1 LDP on the shift space $X$, and in Section~\ref{rate-section} we prove that the rate function is $I$. In Section~\ref{ld} we complete the proof of Theorem~A.

%\textcolor{red}{remove\begin{lemma}\cite[Lemma~3.5]{JaeTak22}\label{meas-rep}
%For any $\mu\in\mathcal M(\mathbb S^1,f)$  
%there exists  $\nu\in\mathcal M(X,\sigma)$ 
%such that $\mu=\nu\circ\pi^{-1}$
%and $h(\mu)=h(\nu)$. Conversely,
% for any $\nu\in\mathcal M(X,\sigma)$, the measure
%$\mu=\nu\circ\pi^{-1}$ belongs to $\mathcal M(\mathbb S^1,f)$ and satisfies
%$h(\mu)=h(\nu)$. 
%\end{lemma}}

 \subsection{Large deviations on the shift space of finite type}\label{ld-shift}
%\textcolor{red}{(Remove) We introduce a {\it free energy} $F\colon\mathcal M(\Lambda(G),f)\to\mathbb R$ by
%$F(\mu)=h(\mu)-\chi(\mu).$ Measures in $\mathcal M(\Lambda(G),f)$ supported on the boundaries of the arcs $[P_i,P_{i+1})$ are supported on periodic orbits, and so they have zero entropy.
%From this and the well-known Ruelle's inequality \cite[Theorem~2]{Rue78b}, $F(\mu)\leq0$ holds for all $\mu\in\mathcal M(\Lambda(G),f)$. Moreover, an ergodic measure $\mu\in\mathcal M(\Lambda(G),f)$ with positive entropy satisfies $F(\mu)=0$ if and only if $\mu$ is absolutely continuous with respect to the Lebesgue measure \cite[Theorem~3]{Led81b}.}

From the proof of Ruelle's Perron-Frobenius theorem in \cite{Bow75}, there exists a Borel probability measure $m$ on $X$ and a constant $C\geq1$ such that
for $\omega\in X$ and $n\geq1$,
\begin{equation}\label{ld-eq5}
\frac{1}{CD_n}\leq\frac{m[\omega_0\cdots \omega_{n-1}]}{\exp\sum_{k=0}^{n-1}\varphi(\sigma^k\omega)}
\leq CD_n.\end{equation}
%Let $m$ denote the Borel probability measure on $X$ determined by the formula
%\[m[\omega_0\cdots \omega_{n-1}]=|\varDelta(\omega_0\cdots \omega_{n-1})|\quad\text{for  $\omega=(\omega_n)_{n=0}^\infty\in X$ and $n\geq1$.}\]
%The measure $m$, which is possibly not $\sigma$-invariant, will be used as a reference measure.
%Using the continuous function $\varphi$ in \eqref{varphi}, 
Here, we have used  that $P(1)=0$, which follows from our assumption that $G$ is of the first kind and   Bowen's formula \cite[Proposition~5.4]{JaeTak22}. 

Note that
the function
\begin{equation}\label{varphi}\varphi=-\log|f'|\circ\pi_X\end{equation}
is continuous,
although $-\log|f'|$ may have discontinuities.
We
define the convex rate function   $I_\varphi\colon\mathbb R\to[0,+\infty]$ by
\[I_\varphi(\alpha)=-\sup\left\{h(\mu)+\int\varphi d\mu\colon\mu\in\mathcal M(X,\sigma), -\int\varphi d\mu=\alpha\right\}.\]
Since $X$ is a finite shift space, the entropy map is upper semicontinuous. 
Hence, $I_\varphi$ is lower semicontinuous.

\begin{prop}[Level-1 LDP]\label{ld-lem}
For any open set $U\subset\mathbb R$
we have
\[\liminf_{n\to\infty}\frac{1}{n}\log m\left\{\omega\in X\colon -\frac{1}{n}\sum_{k=0}^{n-1}\varphi(\sigma^k\omega)\in U \right\}\geq-\inf_{\alpha\in U} I_\varphi(\alpha),\]
and for any closed set $K\subset\mathbb R$ we have
\[\limsup_{n\to\infty}\frac{1}{n}\log m\left\{\omega\in X\colon -\frac{1}{n}\sum_{k=0}^{n-1}\varphi(\sigma^k\omega)\in K\right\}\leq-\inf_{\alpha\in K} I_\varphi(\alpha).\]
\end{prop}
\begin{proof}
Define 
$J\colon\mathcal M(X) \to[-\infty,+\infty]$ by
\[J(\mu)=\begin{cases}
-h(\mu)-\int\varphi d\mu&\text{if $\mu$ is $\sigma$-invariant},\\
+\infty&\text{otherwise.}\end{cases}\]
We have $P(1)=0$, and
\cite[Proposition~3.8]{JaeTak22} implies
\begin{equation}\label{P-new}P(1)=\sup\left\{ h(\mu)+\int\varphi d\mu\colon\mu\in\mathcal M(X,\sigma)\right\}.\end{equation} 
Hence, $J$ is a non-negative function.
For $\omega\in X$ and $n\geq1$, write
$\delta_\omega^n$ for the empirical measure $(1/n)\sum_{k=0}^{n-1}\delta_{\sigma^k\omega}$ in $\mathcal M(X)$, where $\delta_{\sigma^k\omega}$ 
denotes the unit point mass at $\sigma^k\omega.$ 
Let $\tilde\mu_n$ denote the distribution of the random variable $\omega\in X\mapsto\delta_\omega^n\in\mathcal M(X)$ on the Borel probability space $(X,m)$.
%{\sout{This means that $\tilde\mu_n(\mathcal A)=m\{\omega\in X\colon\delta_\omega^n\in\mathcal A\}$ holds for any Borel subset $\mathcal A$ of $\mathcal M(X)$.}}
By virtue of the continuity of the
 map 
$\mu\in\mathcal M(X)\mapsto\int\varphi d\mu\in\mathbb R$
and the
contraction principle \cite{Ell85},
it suffices to 
show
the following level-2 large deviation principle:
\begin{equation}\label{ld-eq2}\liminf_{n\to\infty}\frac{1}{n}\log \tilde\mu_n(\mathcal G)\geq-\inf_{\mu\in \mathcal G} J(\mu)\ \text{ for any open set $\mathcal G\subset\mathcal M(X)$, and }\end{equation}
%and
\begin{equation}\label{ld-eq3}\limsup_{n\to\infty}\frac{1}{n}\log \tilde\mu_n(\mathcal C)\leq-\inf_{\mu\in \mathcal C} J(\mu)\ \text{ for any closed set }\mathcal C\subset\mathcal M(X).\end{equation}

%Let $\nu_n$ denote the measure on $\mathbb R$ that is the push-forward of $\tilde\mu_n$ by the continuous map
%$\Phi\colon\nu\in\mathcal M(X)\mapsto\int\varphi d\nu\in\mathbb R$, 
%namely
%$\nu_n(A)=\tilde\mu_n(\Phi^{-1}(A))=\mu_{\varphi_1}\{\omega\in X
%\colon \delta_\omega^n\in \Phi^{-1}(A)\}$ holds
%for any Borel set $A\subset\mathbb R$.
 % Hence, 

If $G$ has no parabolic element, $\varphi$ is H\"older continuous and the unique equilibrium state  for the potential $\varphi$, denoted by $\mu_{\varphi}$, is a Gibbs state in the sense of Bowen \cite{Bow75}. Moreover,
$\mu_{\varphi}$ is absolutely continuous with respect to $m$ and there exists a constant $c\geq1$ such that $c^{-1}\leq d\mu_{\varphi}/dm\leq c$ $m$-a.e. 
%Hence, for $\omega\in X$ and $n\geq1$ we have
%\begin{equation}\label{ld-eq4}
%m[\omega_0\cdots \omega_{n-1}]\asymp \mu_{\varphi}[\omega_0\cdots \omega_{n-1}]\asymp\exp\sum_{k=0}^{n-1}\varphi(\sigma^k\omega).\end{equation}
Then
 \eqref{ld-eq2} and \eqref{ld-eq3} follow from the results in \cite{Kif90,Tak84}.
If $G$ has a parabolic element, then using
Proposition~\ref{MILD}
%\textcolor{cyan}{If $G$ has a parabolic element,  there is no equilibrium state for the potential $\varphi$ that is absolutely continuous with repect to $m$. Nevertheless, $m$ is a weak Gibbs state \cite{Yur02} in the sense that for all $\omega\in X$ and $n\geq1$,\begin{equation}\label{ld-eq5}\alpha_{\rm min} D_n^{-1}\ll\frac{m[\omega_0\cdots \omega_{n-1}]}{\exp\sum_{k=0}^{n-1}\varphi(\sigma^k\omega)}%\leq \alpha_{\rm max} \ll D_n.\end{equation} Then}
 one can
slightly modify the argument in \cite{Tak84} 
to verify \eqref{ld-eq2} and \eqref{ld-eq3}.
\end{proof}

%Since $m\circ\pi^{-1}$ is the normalized Lebesgue measure on $\mathbb S^1$,
%for any Borel subset $A$ of $\mathbb R$ we have
%\[m\left\{\omega\in X\colon \frac{1}{n}\sum_{k=0}^{n-1}\varphi(\sigma^k\omega)\in A\right\}=\left|\left\{\eta\in\mathbb S^1\colon -\frac{1}{n}\log|(f^n)'\eta|\in A\right\}\right|.\]

\subsection{Identifying the rate function}\label{rate-section}
We now identify the rate function $I_\varphi$ in Proposition~\ref{ld-lem}.
\begin{prop}\label{rate-equal}
We have
$I_\varphi= I$. 
\end{prop}
For a proof of this proposition we need the next lemma.

\begin{lemma}\label{ratefunction}
For all $\alpha\in[\underline\alpha,\overline\alpha]$,
\[I(\alpha)=-\sup\{h(\mu)-\chi(\mu)\colon\mu\in\mathcal M(\mathbb S^1,f),\ \chi(\mu)=\alpha\}.\]
\end{lemma}
\begin{proof}
%Let us copy the proof of \cite[Lemma~2.5]{Tak}. 
Let $\alpha\in[\underline\alpha,\overline\alpha]$.
If $\alpha=0$, then 
 $f$ has a neutral periodic orbit and $\alpha=\underline\alpha$.
Such a periodic orbit supports a periodic measure
with zero Lyapunov exponent. From this, $P(1)=0$, \eqref{P-new} and \cite[Lemma~3.5]{JaeTak22} we have $\sup\{h(\mu)-\chi(\mu)\colon\mu\in\mathcal M(\mathbb S^1,f),\ \chi(\mu)=0\}=0$.
Since $I(0)=0$ by the definition \eqref{rate-f}, the desired equality holds for $\alpha=0$.

Suppose $\alpha>0$.
From the formula \eqref{l-formula},
there exists a sequence $(\nu_n)_{n=1}^\infty$  in $\mathcal M(\mathbb S^1,f)$ such that $\chi(\nu_n)=\alpha$ for $n\geq1$
and $\lim_{n\to\infty}h(\nu_n)/\chi(\nu_n)= b(\alpha)$. Then
\[\begin{split}-\sup\{h(\mu)-\chi(\mu)\colon\mu\in\mathcal M(\mathbb S^1,f),\ \chi(\mu)=\alpha\}&\leq-\lim_{n\to\infty}( h(\nu_n)-\chi(\nu_n))\\
&=\alpha(1-b(\alpha))=I(\alpha).\end{split}\]
On the other hand, for any $\mu\in\mathcal M(\mathbb S^1,f)$ with $\chi(\mu)=\alpha$, we have
\[b(\alpha)\geq\frac{h(\mu)-\chi(\mu)+\chi(\mu)}{\chi(\mu)}
=\frac{h(\mu)-\chi(\mu)+\alpha}{\alpha}.\]
Hence, the reverse inequality holds.
\end{proof}

\begin{proof}[Proof of Proposition~\ref{rate-equal}]
Lemma~\ref{ratefunction} and \cite[Lemma~3.5]{JaeTak22} together imply $I_\varphi(\alpha)\geq I(\alpha)$
for any $\alpha\in[\underline\alpha,\overline\alpha]$. 
Let $\alpha\in(\underline\alpha,\overline\alpha)$.
By the formula \eqref{l-formula}, 
the supremum in the equation in Lemma~\ref{ratefunction}
 is attained. 
By \cite[Lemma~3.5]{JaeTak22},
 for any $\mu\in\mathcal M(\mathbb S^1,f)$  
there exists  $\nu\in\mathcal M(X,\sigma)$ 
such that $\mu=\nu\circ\pi_X^{-1}$
and $h(\mu)=h(\nu)$. This implies
$I_\varphi(\alpha)\leq I(\alpha)$.
Since $I_\varphi$ is lower semicontinuous, convex and $\sup_{\alpha \in (\underline\alpha,\overline\alpha)}I_\varphi(\alpha)<+\infty$, it is continuous on  $[\underline\alpha,\overline\alpha]$.
Since $I$ is continuous on $[\underline\alpha,\overline\alpha]$,  we obtain $I_\varphi(\alpha)=I(\alpha)$
for all $\alpha$ in $[\underline\alpha,\overline\alpha]$.
For all other $\alpha$ we have $I_\varphi(\alpha)=+\infty=I(\alpha)$.
\end{proof}

 \subsection{Proof of Theorem~A}\label{ld} 
 For an open set $U\subset\mathbb R$ and $\varepsilon>0$,
define 
$U_\varepsilon=U\setminus
\left\{x\in\mathbb R\colon\inf_{y\in\mathbb R\setminus U}|x-y|\leq\varepsilon\right\}.$
Note that $U_\varepsilon$ is an open subset of $U$.  
 Using Proposition~\ref{ld-lem}, \eqref{ld-eq5} and Proposition~\ref{MILD} we obtain 
\[\begin{split}-\inf_{\alpha\in U_{\varepsilon}} I_\varphi(\alpha)&\leq\liminf_{n\to\infty}\frac{1}{n}\log m\left\{\omega\in X\colon -\frac{1}{n}\sum_{k=0}^{n-1}\varphi(\sigma^k\omega)\in U_{\varepsilon}\right\}\\
&\leq\liminf_{n\to\infty}\frac{1}{n}\log\left|\left\{\xi\in\mathbb S^1\colon \frac{1}{n}\log|(f^n)'\xi|\in U_{\varepsilon/2}\right\}\right|\\
&\leq\liminf_{n\to\infty}\frac{1}{n}\log\left|\mathscr{H}_n(U)\right|.\end{split}\]
Similarly, for a closed set $F\subset\mathbb R$ and $\varepsilon>0$, define
$F^\varepsilon=\left\{x\in\mathbb R\colon \inf_{y\in F}|x-y|\leq\varepsilon\right\}.$
Note that $F^\varepsilon$ is a closed set containing $F$.
Using Proposition~\ref{ld-lem}, \eqref{ld-eq5} and Proposition~\ref{MILD} we obtain 
%A similar reasoning shows that for any closed set $F\subset\mathbb R$ and any $\varepsilon>0$, 
\[\begin{split}-\inf_{\alpha\in F^\varepsilon } I_\varphi(\alpha)&\geq\limsup_{n\to\infty}\frac{1}{n}\log m\left\{\omega\in X\colon -\frac{1}{n}\sum_{k=0}^{n-1}\varphi(\sigma^k\omega)\in F^\varepsilon\right\}\\
&\geq\limsup_{n\to\infty}\frac{1}{n}\log\left|\left\{\xi\in\mathbb S^1\colon \frac{1}{n}\log|(f^n)'\xi|\in F^{\varepsilon/2}\right\}\right|\\
&\geq\limsup_{n\to\infty}\frac{1}{n}\log\left|\mathscr{H}_n(F)\right|.\end{split}\]
%From Lemma~\ref{rate-equal},
Since $U$ is open and $F$ is closed,
 the lower semicontinuity of $I_\varphi$ implies
 $\inf_{U_\varepsilon} I_\varphi\to\inf_U I_\varphi$ and
  $\inf_{F^\varepsilon} I_\varphi\to
  \inf_F I_\varphi$ as $\varepsilon\to0$.
Hence we obtain
\[\liminf_{n\to\infty}\frac{1}{n}\log|\mathscr{H}_n(U)|\geq-\inf_{\alpha\in U }
I_\varphi(\alpha)\text{ and }\limsup_{n\to\infty}\frac{1}{n}\log|\mathscr{H}_n(F)|\leq-\inf_{\alpha\in F }
I_\varphi(\alpha).\]
The equality in (a) of Theorem~A follows from combining the above two estimates with Proposition~\ref{rate-equal}.

We now turn to the proof of (b).  The definition of $I$ in \eqref{rate-f} gives $I^{-1}(+\infty)=\mathbb R\setminus[\underline\alpha,\overline\alpha]$. The continuity of $I$ on $[\underline\alpha,\overline\alpha]$ and the analyticity of $I$ on $(\underline\alpha, \overline\alpha)$ follow from 
Theorem~\ref{bslyapunov}(b).  Differentiating the formula \eqref{rate-f} gives
$I'(\alpha)=1-b(\alpha)-\alpha b'(\alpha)$.
By Theorem~\ref{bslyapunov}(c) we have
$\lim_{\alpha\nearrow\overline\alpha}b'(\alpha)=-\infty$. Hence, 
$\lim_{\alpha\nearrow\overline\alpha}I'(\alpha)=+\infty$.
If $G$ has no parabolic element, then   $\lim_{\alpha\searrow\underline\alpha}b'(\alpha)=+\infty$ by Theorem~\ref{bslyapunov}(c). Hence,
$\lim_{\alpha\searrow\underline\alpha}I'(\alpha)=-\infty.$ 

\begin{lemma}[\cite{JaeTak22}, Section~5]\label{beta-alpha}
    There exists a strictly decreasing analytic function $\beta\colon(\underline\alpha,\overline\alpha)\rightarrow (-\infty,\beta_+) $ satisfying 
$-P'(\beta(\alpha))=\alpha$.
We have
\begin{equation}\label{limbeta}\lim_{\alpha\searrow
\underline\alpha}\beta(\alpha)=\beta_+\ \text{ and }\
\lim_{\alpha\nearrow\overline\alpha}\beta(\alpha)=-\infty.\end{equation}
 For all $\alpha\in (\underline\alpha,\overline\alpha)$ we have
\begin{equation} \label{eq-spectrum1}
\alpha b(\alpha)=P(\beta(\alpha))+ \alpha\beta(\alpha)
\end{equation}
and
\begin{equation} \label{eq-spectrum2}
b'(\alpha)=\frac{-P(\beta(\alpha))}{\alpha^2}.
\end{equation}
\end{lemma}

%{\tiny  \begin{lemma}\label{powerseries} Let $f\colon I\to\mathbb R$ be a monotone, convex  analytic function. If $f$ is not affine, then $f''>0$.\textcolor{red}{$I$, $f$ are not good.}\end{lemma} \begin{proof} Let $x_0\in I$. Assume $f''(x_0)=0$. By the convexity of $f$, $f''\geq0$ and so $f'''(x_0)=0$. If $f''''(x_0)\neq0$, then $x_0$ is a local maximal or minimal point of $f$, a contradiction. So, $f''''(x_0)=0$. If $f'''''(x_0)=(f'')'''(x_0)\neq0$, then $f''$ becomes negative near $x_0$, a contradiction. So, $f'''''(x_0)=0$. If $f^{(6)}(x_0)\neq0$, then $x_0$ is a local minimal or maximal point of $f$, a contradiction. So  $f^{(6)}(x_0)=0$. In this way, we get $f^{(n)}(x_0)=0$ for all $n\geq2$. By the analyticity of $f$, $f$ becomes an affine function. a contradiction. \end{proof}}

First assume that $G$ has a parabolic element. Recall that in this case, $\beta_+=1$. 
Differentiating formula
\eqref{eq-spectrum1} gives
$\beta(\alpha)=b(\alpha)+\alpha b'(\alpha).$ Combining this with $\lim_{\alpha\searrow\underline\alpha}\beta(\alpha)=1$ from \eqref{limbeta}, and $\lim_{\alpha\searrow\underline\alpha}b(\alpha)=b(0)=1$ from Theorem~\ref{bslyapunov}(c), we obtain  $\lim_{\alpha\searrow\underline\alpha}\alpha b'(\alpha)=0$, and so
$\lim_{\alpha\searrow\underline\alpha}I'(\alpha)=0$.
 Since $\lim_{\alpha\nearrow\overline\alpha}I'(\alpha)=+\infty$,  Lemma~\ref{powerseries} below applied to $I$ shows that $I''>0$ on $(\underline\alpha,\overline\alpha)$.

Now assume that $G$ has no parabolic element. Applying Lemma~\ref{powerseries} to the restrictions of $I$ to $(\underline\alpha,\alpha_G)$ and $(\alpha_G,\overline\alpha)$ shows that $I''>0$ on $(\underline\alpha,\overline\alpha)\setminus \{\alpha_G \}$. Here, we have used again that $I$ is not affine because $|I'(\alpha)|$ tends to infinity as $\alpha$ tends to the boundary of $(\underline\alpha,\overline\alpha)$. To prove $I''(\alpha_G)>0$ we differentiate \eqref{eq-spectrum2} to get
\[
b''(\alpha_G)=\frac{P'(\beta(\alpha_G))}{\alpha_G^2}\frac{1}{P''(\beta(\alpha_G))},
\]
and use
$I''(\alpha_G)=-2b'(\alpha_G) - \alpha_G b''(\alpha_G)=-\alpha_G b''(\alpha_G)$
and $P''>0$ in \cite[Proposition~5.7]{JaeTak22}. The proof of Theorem~A is complete.
\qed

\begin{lemma}\label{powerseries}
 Let $h\colon (a,b)\to\mathbb R$
be a monotone, convex  analytic function. Then either $h''>0$  on $(a,b)$, or $h$ is affine. \end{lemma}
\begin{proof}
We assume that $h''(x)=0$ for some $x\in(a,b)$.
Let $k=\inf \{k> 2\colon h^{(k)}(x)\neq 0 \}$. 
  If $k$ is finite and even, then $h$ has a local extremum at $x$ contradicting the monotonicity of $h$. If $k$ is finite and odd, then $h''$ changes the sign at $x$  contradicting the convexity of $h$. It follows that $k$ is infinite. Since 
    $h$ is analytic it is affine.  
\end{proof}

\section{Refined large deviations upper bounds}
This last section is dedicated to the proof of Theorem~B. In Section~\ref{c-ind} we prove the existence of uniform distortion bounds associated with an induced map constructed in \cite{JaeTak22}. In Section~\ref{sec-est} we extract finite subsystems and develop some estimates on them. In Section~\ref{pfthmb} we complete the proof of Theorem~B.
 \subsection{Bounded distortion from the induced Markov map}\label{c-ind}
Let $f$ be the Bowen-Series map with the finite Markov partition $(\varDelta(a))_{a\in S}$ constructed in Section~\ref{markov-sec}.
Define 
 \[\tilde\varDelta=\mathbb S^1\setminus\left( V_c\cup f^{-1}(V_c)\cup  \bigcup_{v\in V_c}L(v)\cup R(v)\right).\]
% \textcolor{red}{The same correction to be done on \cite{JaeTak22} YES!}
 Note that $\tilde\varDelta$ is a non-empty set. %\textcolor{blue}{We should remove *all* preimages of $V_c$ to make \eqref{inducedmap} well defined.
 %(H) Preimages are dense, so removing all them contradicts the assumption of the Markov map. Maybe, only removing $f^{-1}(V_c)$ suffices: 
%If $k\geq2$, $\xi\in f^{-k}(V_c)\cap\tilde\varDelta$
 %and $f(\xi),\ldots,f^{k-1}(\xi)\notin V_c$, then $t(\xi)$ is finite.}
% We assume $p$ is large enough so that
% $\tilde\varDelta$ is non-empty \textcolor{red}{Maybe $\tilde\varDelta\neq\emptyset$ for $p=1$.}
Define  $t\colon \tilde\varDelta\to\mathbb N$
by %\textcolor{red}{(remove)\begin{equation}\label{t-def}t(x)=\inf\left(\{n\geq1\colon f^{n}(x)\in \tilde\varDelta\}\cup\{+\infty\}\right).\end{equation}}
\[t(\xi)=\inf\{n\geq1\colon f^{n}(\xi)\in \tilde\varDelta
\}.\]
%If $t(x)$ is finite, it is the first return time of the $f$-orbit of $x$ to $\tilde\varDelta$.
%Set
%\[\tilde\Lambda=\tilde\varDelta-\bigcup_{n=0}^{\infty} f^{-pn}(\{t=\infty\} ),\]
%\textcolor{red}{(remove) Then $t(x)=\infty$ if and only if $x\in \bigcup_{n=1}^\infty f^{-n}(V_c)$.}
Define an induced map
\[\tilde{f}: \tilde\varDelta \to \mathbb{S}^1,\quad \xi\mapsto f^{t(\xi)}(\xi).\]
%\textcolor{red}{(remove) and set
%\begin{equation}\label{induce-eq10}\tilde\Lambda=\bigcap_{n=0}^{\infty} \tilde f^{-n} (\tilde\varDelta).\end{equation}}
% Since $\tilde\varDelta$ is the finite union of
% elements of 
% $(\varDelta_\omega)_{\omega\in E^{p}(X)}$ that is
%a Markov partition for $f^{p}$,
Replacing each $\varDelta(a)$, $a\in S$,  by the countably many arcs on which $t$ is finite and constant, we obtain a countably infinite subset $\tilde S$ of $E(X)$ such that $\tilde \varDelta=\bigcup_{\tilde a\in \tilde S}\tilde \varDelta(\tilde a)$ and a Markov map
$\tilde{f}$ with a Markov partition 
$(\tilde \varDelta(\tilde a))_{\tilde a\in \tilde S}$. This determines by \eqref{m-shift} a countable Markov shift
\[\tilde X=\tilde X(\tilde f,(\tilde \varDelta(\tilde a))_{\tilde a\in \tilde S}).\]

% Pollicott \cite{Pol91} and Morita \cite{Mor97} used the Bowen-Series maps to establish analogues of `prime number theorem' 
%on growth rates of the number of closed geodesics of hyperbolic surfaces. 
%We use the Main Theorem to describe asymptotic behaviors of non-closed geodesics.

%Put \[C_0=\min_{a\in S}|f(\varDelta(a))|.\]
%\textcolor{green}{We can do without this lemma. Proposition~\ref{MILD} suffices.
%Then we need more subexponential multiplicative factor.} \textcolor{blue}{JJ: to get the $n^2$-bound in the theorem for groups with cusp, we need this uniform? YES.}
%\begin{lemma}\label{good-dist}
%There exists a constant $C\geq1$ such that if $a\in S$, $n\geq1$ and $\omega=\omega_0\cdots\omega_{n}\in E^{n+1}(X)$ satisfy \textcolor{red}{$\varDelta(a)\subset\tilde\varDelta$ (strong assumption?)} and $\omega_{n}=a$, then for all $\xi,\eta\in\varDelta(\omega)$ we have \[\frac{|(f^n)'\xi|}{|(f^n)'\eta|}\leq C.\] \end{lemma}
%\begin{proof}The assumption of the lemma implies that there exists $k\leq n$ and $\tilde\omega\in E^k(\tilde X)$ such that $\xi,\eta\in\tilde\varDelta(\tilde\omega)$. The conclusion of the lemma follows from the bounded distortion of $\tilde f$ in \cite[Formula (4.6) in the proof of Proposition 4.4]{JaeTak22}.
%\textcolor{red}{$\xi,\eta\in\tilde X$} \end{proof}

\begin{lemma}\label{good-dist}
There exists a constant $C_1\geq1$ such that  if  $N\ge 1$, $\tilde\omega_0\cdots \tilde\omega_{N-1}\in E^{N}(\tilde X)$, $\omega\in E(X)$ satisfy $\tilde\omega_0\cdots \tilde\omega_{N-1}= \omega$,
then for all 
$\xi,\eta \in\varDelta(\omega)$
  we have
\[\frac{|( f^{|\omega|})'\xi|}{|( f^{|\omega|})'\eta|}\leq C_1.\]
\end{lemma}
\begin{proof}
Recall that 
$\tilde\varDelta(\tilde{\omega}_{0}\dots\tilde{\omega}_{N-1})=\varDelta(\omega)$.
For $\xi,\eta \in \tilde\varDelta(\tilde{\omega}_{0}\dots\tilde{\omega}_{N-1}) \cap \Lambda_c$ the estimate was verified in the proof of
\cite[(4.6) in Proposition~4.4]{JaeTak22}. Since $G$ is of the first kind, $\Lambda_c$ is dense in 
$\mathbb{S}^1$. Since  $(\tilde f)^N= f^{|\omega|}$ on $\varDelta(\omega)$, the desired estimate follows from the continuity of $(f^{|\omega|})'$ on $\varDelta(\omega)$.
\end{proof}

\subsection{Estimates on finite subsystems}\label{sec-est}

%We construct a finite subsystem by choosing bridges.
Fix $a_*\in \tilde S$. Then
\[\tilde\varDelta(a_*)=\varDelta(a_*)\subset\tilde\varDelta.\]
Using the finite irreducibility of $f$,
for each $a\in S$ we 
fix words 
 $\lambda(a)$,
$\rho(a)$ in $E(X)\cup E^0$ 
such that $a_*\lambda(a) a$, $a\rho(a) a_*\in E(X)$, namely
$\varDelta(a)\subset f^{|a_*\lambda(a)|}(\varDelta(a_*))$ and 
$\varDelta(a_*)\subset f^{|a\rho(a)|}(\varDelta(a))$.
For each $\omega\in E(X)$ of word length $n\geq1$, put 
\[\hat{\omega}=a_*\lambda(\omega_0)\omega
\rho(\omega_{n-1})a_*.\]

\begin{lemma}\label{lem-loop}
There exists a constant $C_2>0$ such that
for every $\omega\in E(X)$,
\[
|\varDelta(\omega)|\leq C_2|\varDelta(\hat\omega)|.
\]
\end{lemma}

\begin{proof}
Let $n\geq1$, $\omega\in E^n(X)$
and write $\omega=\omega_0\cdots\omega_{n-1}$. We have
\[
\frac{|\varDelta(\hat\omega )|}{|\varDelta(\omega
\rho(\omega_{n-1})a_*)|}\geq\frac{|\varDelta(\hat\omega )|}
{|f^{|a_*\lambda(\omega_0)|} (\varDelta(\hat\omega ))|}\geq \frac{1}{\sup_{\varDelta }|(f^{|a_*\lambda(\omega_0)| })'|}.\]
The first inequality follows from 
$f^{|a_*\lambda(\omega_0)|} (\varDelta(\hat\omega))\supset \varDelta(\omega\rho(\omega_{n-1})a_*)$,
and the second one from
the mean value theorem applied to the restriction of
$f^{|a_*\lambda(\omega_0)|}$
to $\varDelta(\hat\omega )$.
We also have
\[\frac{|\varDelta(\omega\rho(\omega_{n-1})a_*)|}{|\varDelta(\omega)|}\geq \frac{1}{C_1}\frac{|f^n(\varDelta(\omega\rho(\omega_{n-1})a_*))|}{|f^n(\varDelta(\omega))|}
\geq \frac{1}{2\pi C_1}|\varDelta(a_*)|.\]
The first inequality follows from Lemma~\ref{good-dist}, and
the second one follows from
$f^n(\varDelta(\omega\rho(\omega_{n-1})a_*))=\varDelta(a_*)$
together with $|f^n(\varDelta(\omega))|\leq |\mathbb S^1|=2\pi$.
Combining these two estimates and
setting $C_2=2\pi C_1\sup_{a\in S}\sup_{\mathbb S^1 }|(f^{|a_*\lambda(a)| })'|/|\varDelta(a_*)|$
we obtain the desired inequality.\end{proof}

For $\alpha\in\mathbb R$ and
  $n\geq1$ define
  \[L_n(\alpha)=
\left\{\omega\in E^n(X)\colon  %\varDelta(\omega)\subset\varDelta_n,\ 
\sup_{\xi\in\varDelta(\omega)}\log|(f^n)'\xi|\geq\alpha n\right\}.\] 
Set \[N_0=\max_{a\in S}|\lambda(a)|+\max_{a\in S}|\rho(a)|.\]
%\quad\text{ and }\quad C=\frac{(\sup_{\varDelta}|f'|)^{n_0}}{ |\varDelta(a_*)|}.\]
For each  $q\in\{n+2|a_*|,\ldots, n+N_0+2|a_*|\}$,
set 
\[L_n(\alpha,q)=\{\omega\in L_n(\alpha)\colon \hat\omega\in E^q(X)\}.\]
Clearly we have $L_n(\alpha)=\bigcup_{q=n+2|a_*|}^{n+N_0+2|a_*|}L_n(\alpha,q).$
Put $M=\max_{\mathbb S^1}\log|f'|.$

 \begin{lemma}\label{horse}
 If 
 $L_n(\alpha,q)\neq\emptyset$, then
 for any $\varepsilon>0$
 there exists a measure $\mu_\varepsilon\in\mathcal M(\mathbb S^1,f)$ such that 
\[ \sum_{\omega\in L_n(\alpha,q)}
|\varDelta(\hat\omega)|\leq C_1\exp((h(\mu_\varepsilon)-\chi(\mu_\varepsilon))n+\varepsilon)\quad\text{and}\]
\[
 \chi(\mu_\varepsilon)\geq\frac{n}{q}\alpha -\frac{q-n}{q}M-\frac{C_1}{q}.\]
 \end{lemma}
 \begin{proof}
  Put
   $p=\#L_n(\alpha,q)$,
$\{\varDelta(\hat\omega)\colon\omega\in L_n(\alpha,q)\}=\{\varDelta_i\}_{i=1}^p$, $W=\bigcup_{i=1}^p
 \varDelta_i$ and
    $F=f^q|_W$.
% $K=\bigcap_{j=0}^{\infty}F^{-j}({\rm cl}(\bigcup_{i=1}^p\varDelta_i))$ and
 % \begin{sublemma}  The arcs ${\rm cl}(\varDelta(\hat\omega))$,  $\omega\in L_n(\alpha,q)$ are pairwise disjoint. \end{sublemma}\begin{proof} For any $\omega\in L_n(\alpha,q)$ we have $f^{q-1}(\varDelta(\hat\omega))=\varDelta(a_*)$. Let  $\omega,\omega'\in L_n(\alpha,q)$, $\omega\neq\omega'$ and suppose  ${\rm cl}(\varDelta(\hat\omega))\cap {\rm cl}(\varDelta(\hat\omega))\neq\emptyset$. If $f^{q-1}$ is continuous on ${\rm cl}(\varDelta(\hat\omega))\cup {\rm cl}(\varDelta(\hat\omega))$, then we reach a contradiction.
% \end{proof}
 The left-closed right open arcs $\varDelta_i$, $1\leq i\leq p$
  are pairwise disjoint,
  and their $F$-images contain $W$. 
  Hence, $F$  
 is a Markov map with a Markov partition $(\varDelta_i)_{i\in\{1,\ldots,p\} }$ for which the associated transition matrix has no zero entry.
 %From this and
 %the decay of cylinders for $f$ in \cite[Lemma~2.5]{JaeTak22}, it follows that  %$F|_K\colon K\to K$ is topologically semiconjugate to the full shift on $p$ symbols.

Let $\Sigma_p=\{1,\ldots,p\}^\mathbb N$
denote the full shift space on $p$ symbols and
let $\sigma_p$ denote the left shift acting on $\Sigma_p$.
 The coding map $
 \pi_{\Sigma_p}\colon\Sigma_p\to \mathbb S^1$ satisfies
  $F\circ\pi_{\Sigma_p}= \pi_{\Sigma_p}\circ\sigma_p$
  on $\pi_{\Sigma_p}^{-1}(W)$. If $\nu$ is a non-atomic   $\sigma_p$-invariant Borel probability measure on $\Sigma_p$, then $\nu\circ\pi_{\Sigma_p}^{-1}$ is an $F$-invariant Borel probability measure and %$h_{\nu\circ\pi_{\Sigma_p}^{-1}}(F)=h_{\nu}(\sigma_p)$.\textcolor{red}{triple subscript 
  $h(\nu)=h(\nu\circ\pi_{\Sigma_p}^{-1})$.
 The function
$\Phi=-\log |F'\circ\pi_{\Sigma_p}|$ on
$\Sigma_p\cap\pi_{\Sigma_p}^{-1}(W)$
 is uniformly continuous.
Since $\Sigma_p\cap\pi_{\Sigma_p}^{-1}(W)$
is a dense subset of $\Sigma_p$, $\Phi$ admits a unique  continuous extension to $\Sigma_p$ which we still denote by $\Phi$.
The variational principle gives %\cite[Lemma~1.20]{Bow75} 
 \begin{equation}\label{equation}
 \begin{split}\sup_{\nu\in\mathcal M(\Sigma_p,\sigma_p)}\left( h(\nu)+\int\Phi d\nu\right)&=
 \lim_{j\to\infty}\frac{1}{j}\log\left(\sum_{x\in \sigma_p^{-j}(p^\infty)}
 \exp{\sum_{k=0}^{j-1}
 \Phi(\sigma_p^k(x))}\right),\end{split}\end{equation}
where $p^\infty$ denotes the fixed point of $\sigma_p$ in the $1$-cylinder $[p]$.
 
For each $\omega\in L_n(\alpha,q)$ there exist $N=N(\omega)\ge 1$ and $\tilde\omega_0\cdots \tilde\omega_{N-1}\in E^{N}(\tilde X)$ such that $\tilde\omega_0\cdots \tilde\omega_{N-1}=\hat\omega$.
By Lemma~\ref{good-dist}, for each $i\in\{1,\ldots,p\}$ and all 
 $\xi$, $\eta\in\varDelta_i$, \begin{equation}\label{distortion-eq}\frac{|F'\xi|}{|F'\eta|}=\frac{|(f^{q})'\xi|}{|(f^{q})'\eta|}
 \leq  C_1.\end{equation}
%It follows that for any $j\geq1$,any $i_0\cdots i_{j-1}\in\{1,\ldots,p\}^j$ and for all  $x,y\in[i_0\cdots i_{j-1}]\subset\Sigma_p$ we have\[  \sum_{k=0}^{j-1} \Phi(\sigma^k(x))-\sum_{k=0}^{j-1} \Phi(\sigma^k(y))\leq\log C_0.\]

 For the sum inside the logarithm in \eqref{equation},
 using \eqref{distortion-eq} we have 
 %\textcolor{red}{(remove)\[\begin{split}\label{sarf}\sum_{x\in (F|_K)^{-j}(y_0)}&\exp\left(\sum_{k=0}^{j-1}\Phi(F^k(x))\right)\geq\left(\inf_{y'\in K}\sum_{x\in (F|_K)^{-1}(y')}e^{\Phi(x)}\right)^j\\&\geq\left(\sum_{\omega\in L_n(\alpha,q)}\inf_{\varDelta(\hat\omega)}e^{\Phi}\right)^j\geq\left(C_0^{-1}\sum_{\omega\in L_n(\alpha,q)}|\varDelta(\hat\omega)|\right)^j.\end{split}\]}
\[
 \begin{split}\label{sarf}\sum_{x\in \sigma_p^{-j}(p^\infty)}&\exp\left(\sum_{k=0}^{j-1}\Phi(\sigma_p^k(x))\right)
\geq\left(\inf_{y'\in \Sigma_p}\sum_{x\in \sigma_p^{-1}(y')}e^{\Phi(x)}\right)^j\\
&\geq\left(\sum_{i=1}^p\inf_{\varDelta_i}
(-\log|F'|)\right)^j\geq\left(\frac{1}{C_1}\sum_{i=1 }^p|\varDelta_i|\right)^j.
\end{split}\]
Taking logarithm on both sides, dividing by $j$ and letting $j\to\infty$ we get
%\textcolor{red}{(remove) \[\lim_{j\to\infty}\frac{1}{j}\log\left(\sum_{x\in (F|_K)^{-j}(y_0)}\exp\sum_{k=0}^{j-1}\Phi(F^k(x))\right)\geq\log\sum_{\omega\in L_n(\alpha,q)}|\varDelta(\hat\omega)|-\log C_0.\]}
\[\lim_{j\to\infty}\frac{1}{j}\log\left(\sum_{x\in \sigma_p^{-j}(p^\infty)}
\exp\sum_{k=0}^{j-1}\Phi(\sigma_p^k(x))\right)
\geq\log\sum_{i=1 }^p|\varDelta_i|-\log C_1.\]
Plugging this into the previous inequality yields
%\textcolor{red}{(remove)\begin{equation}\label{sup-eq}\sup_{\nu\in\mathcal M(K,F|_K)}\left( h(\nu)+\int\Phi d\nu\right)\geq\log\sum_{\omega\in L_n(\alpha,q)}|\varDelta(\hat\omega)|-\log C_0.\end{equation}}
\begin{equation}\label{sup-eq}\sup_{\nu\in\mathcal M(\Sigma_p,\sigma_p)}\left( h(\nu)+\int\Phi d\nu\right)\geq
\log\sum_{i=1}^p|\varDelta_i|-\log C_1.\end{equation}

\begin{sublemma}\label{non-atomic}
We have 
\[\sup_{\stackrel{\nu\in\mathcal M(\Sigma_p,\sigma_p)}{{\rm non-atomic}}}\left(h(\nu)+\int\Phi d\nu\right)=\sup_{\nu\in\mathcal M(\Sigma_p,\sigma_p)}\left( h(\nu)+\int\Phi d\nu\right).\]
%For any $\varepsilon>0$, there exists a non-atomic measure $\nu_\varepsilon\in\mathcal M(\Sigma_p,\sigma)$ such that  $h(\nu)+\int\Phi d\nu>\sup_{\nu\in\mathcal M(\Sigma_p,\sigma)}\left( h(\nu)+\int\Phi d\nu\right)-\varepsilon$.
\end{sublemma}
\begin{proof}
Take $\nu_0\in\mathcal M(\Sigma_p,\sigma_p)$ which attains the supremum of the right-hand side. 
Take $\nu_1\in\mathcal M(\Sigma_p,\sigma_p)$ with positive entropy. For any $s\in(0,1]$
the measure $\nu_s=(1-s)\nu_0+s\nu_1$ has positive entropy. 
Since $\sigma_p$-invariant ergodic   measures are entropy-dense \cite{eizen94}, for any $t>0$ there is $\nu_{s,t}\in\mathcal M(\Sigma_p,\sigma_p)$ which is ergodic, 
has positive entropy %$h(\nu_{s,t})>0$ 
and hence is non-atomic,
and satisfies $h(\nu_{s,t})+\int\Phi d\nu_{s,t}> h(\nu_s)+\int\Phi d\nu_s-t.$ Since $s$ and $t$ are arbitrary,
we obtain the desired equality.
\end{proof}
Let $\varepsilon>0$, and take a non-atomic measure $\nu_\varepsilon\in\mathcal M(\Sigma_p,\sigma_p)$ such that
\[h(\nu_\varepsilon)+\int\Phi d\nu_\varepsilon>\sup_{\nu\in\mathcal M(\Sigma_p,\sigma_p)}\left( h(\nu)+\int\Phi d\nu\right)-\varepsilon.\]
 The measure
 $\nu_\varepsilon\circ\pi_{\Sigma_p}^{-1}$ is $F$-invariant, and
the measure
$\mu_\varepsilon= (1/q)\sum_{j=0}^{q-1}\nu_\varepsilon\circ\pi_{\Sigma_p}^{-1}\circ f^{-j}$ belongs to $\mathcal M(\mathbb S^1,f)$, and by \eqref{sup-eq} satisfies
\[\begin{split}\log\sum_{i=1}^p|
\varDelta_i|-\log C_1&\leq h(\nu_\varepsilon)+\int\Phi d\nu_\varepsilon+\varepsilon\\
&= (h(\mu_\varepsilon)-\chi(\mu_\varepsilon))q+\varepsilon
\leq (h(\mu_\varepsilon)-\chi(\mu_\varepsilon))n+\varepsilon,\end{split}\]
where the last inequality follows from $q\geq n$ and $P(1)=0$.
We have verified the first inequality of Lemma~\ref{horse}.

By \eqref{distortion-eq} and
the definition of $L_n(\alpha)$, 
for each $i\in\{1,\ldots,p\}$ we have
\[\inf_{\varDelta_i }\log|F'|\geq
\sup_{\varDelta_i } \log|F'|-C_1\geq \alpha n-(q-n)M -C_1.\]
Since
$\chi(\mu_\varepsilon)q=\int\log|F'| d\nu_\varepsilon\circ\pi_{\Sigma_p}^{-1}$
and the support of $\nu_\varepsilon\circ\pi_{\Sigma_p}^{-1}$ is contained in the closure of $\bigcup_{i=1 }^p\varDelta_i$, the second inequality in Lemma~\ref{horse} follows.
 \end{proof}

\subsection{Proof of Theorem~B}\label{pfthmb}
We only give a proof of Theorem~B(b) since that of  Theorem~B(a) is analogous. 
Suppose $G$ has a parabolic element.
Let $K$ be a compact neighborhood of $0$ in $\mathbb D$. 
 Let $\alpha\in(0,\overline\alpha)$. 
 For each $n\geq1$ put
 \[\alpha_n=\alpha-\frac{2\log n+C_0}{n},\]
 where $C_0>0$ is the constant in Proposition~\ref{lem-logdist}.
Then we have
\begin{equation}\label{eq-pfthmb}\mathscr{H}_n([\alpha,+\infty),K)\subset\left\{\xi\in\mathbb S^1\colon \log|(f^n)'\xi|\geq \alpha_nn\right\}\subset L_n(\alpha_n).\end{equation}
In order to  estimate $|L_n(\alpha_n)|$,  our strategy is to
choose a family of periodic admissible words  
of length approximately $n$, 
construct a finite subsystem and then
  construct an $f$-invariant measure
  using the thermodynamic formalism.
 
Choose $q_*\in\{n+2|a_*|,\ldots,n+N_0+2|a_*|\}$ which maximizes the function $q\in\{n+2|a_*|,\ldots,n+N_0+2|a_*|\}\mapsto\sum_{\omega\in L_n(\alpha_n,q)}|\varDelta(\hat\omega)|.$
Lemma~\ref{lem-loop} gives
 \begin{equation}\label{refine1}\begin{split}
    \sum_{\omega\in L_n(\alpha_n)}|\varDelta(\omega)|&\leq
    C_2\sum_{\omega\in L_n(\alpha_n)}|\varDelta(\hat\omega)|\leq
C_2(N_0+1)\sum_{\omega\in L_n(\alpha_n,q_*)}|\varDelta(\hat\omega )|.
\end{split}\end{equation}
By Lemma~\ref{horse}, for any $\varepsilon>0$ there exists
$\mu_\varepsilon\in\mathcal M(\mathbb S^1,f)$ such that 
\begin{equation}\label{horse1}  \sum_{\omega\in L_n(\alpha_n,q_*)}
|\varDelta(\hat\omega)|\leq C_1\exp((h(\mu_\varepsilon)-\chi(\mu_\varepsilon))n+\varepsilon),\text{ and }\end{equation} 
\begin{equation}\label{horse2}
\begin{split}
\chi(\mu_\varepsilon)&\geq\frac{n}{q_*}\alpha_n -\frac{q_*-n}{q_*}M -\frac{C_1}{q_*}\\
&=\alpha-\frac{q_*-n}{q_*}\alpha-\frac{2\log n+C_0}{q_*}-\frac{q_*-n}{q_*}M -\frac{C_1}{q_*}\\
&\geq\alpha-\frac{N_0+2|a_*|}{n}\left(\alpha+M \right)-\frac{2\log n+C_0+C_1}{n}\\&\geq\alpha-\delta_n,\end{split}\end{equation}
where
\[\delta_n=\frac{N_0+2|a_*|}{n}\left(\overline\alpha+M \right)+\frac{2\log n+C_0+C_1}{n}.\]
%\begin{lemma}\label{coding}
%Let $n\geq1$ be an integer and set $\hat T=T^n$. Let $N\in\mathbb Z_{+}$ and let %$\{\omega^{(i)}\}_{i=1}^N$ be a finite subset of $E^{n+1}(a_0)$.
%Then for each $\eta=(\eta_k)_{k=0}^\infty\in\Sigma_N$ the set
%\[\bigcap_{k=0}^{\infty} \overline{\hat T^{-k}(\varDelta_{\omega^{\eta_k}})}\]
%is a singleton, and 
%the map $\pi\colon\Sigma_N\to\bigcap_{k=0}^\infty \overline{\hat T^{-k}(\bigcup_{i=1}^N \varDelta_{\omega^i})}$ defined by $\pi(\omega)\in\bigcap_{k=0}^{\infty}\overline{ \hat T^{-k}(\varDelta_{\omega^{\eta_k}})}$ 
%is a homeomorphism onto its image satisfying $\hat T\circ\pi=\pi\circ\sigma$.
%\end{lemma}

There exists a constant $\kappa_0>0$ which is independent of
 $\alpha$ such that
if $n\geq1$ and $n\geq\max\{\kappa_{0}/\alpha,\min\left\{n\geq1\colon (1/n)\log n\leq\alpha/3\right\}\}$ then 
 $\alpha-\delta_n>0.$
 We have
 \[ \exp((h(\mu_\varepsilon)-\chi(\mu_\varepsilon))n)
\leq  \exp(-I(\chi(\mu_\varepsilon))n)\leq\exp\left(-I(\alpha-\delta_n)n\right).
\]
The first inequality is by Lemma~\ref{ratefunction},
and the second one is by \eqref{horse2} and the fact that $I$ is monotone increasing on $(0,\overline\alpha]$.
By the mean value theorem, there exists
$\theta\in[\alpha-\delta_n,\alpha]$ such that
\[I(\alpha)-I\left(\alpha-\delta_n\right)=I'(\theta)\delta_n\leq I'(\alpha)\delta_n,\]
where the last inequality follows from the monotonicity of $I'$ 
on $(0,\overline\alpha]$
which is a consequence of the convexity of $I$.
Plugging this inequality into the previous one yields
  \begin{equation}\label{refine2}  \exp((h(\mu_\varepsilon)-\chi(\mu_\varepsilon))n)
\leq n^2e^{C_2I'(\alpha)}e^{-I(\alpha)n},
\end{equation}
where
\[C_3=(N_0+2|a_*|)\left(\overline\alpha+M \right)+C_0+C_1.\]
Combining \eqref{eq-pfthmb}, \eqref{refine1}, \eqref{horse1} and \eqref{refine2} we obtain
\[
\begin{split}
   |\mathscr{H}_n([\alpha,+\infty),K)|\leq
C_1C_2(N_0+1)e^{C_3I'(\alpha)}e^{-I(\alpha)n}e^\varepsilon.\end{split}\]
Put $\kappa_1=C_1C_2(N_0+1)$ and $\kappa_2=e^{C_3}$. Since $\varepsilon>0$ is arbitrary, this implies
 the desired bound in Theorem~B(b).    
\qed

\subsection*{Acknowledgments}
JJ was supported by the JSPS KAKENHI 21K03269.
HT was supported by the JSPS KAKENHI 
19K21835 and 20H01811.


\begin{thebibliography}{10}
   
   
%   \bibitem{AdlFla91} Adler, R., Flatto, L.: Geodesic flows, interval maps, and symbolic dynamics.  Bull. Am. Math. Soc. New Ser. {\it  25} 229--334 (1991)
   

  
  
   
%   \bibitem{AdlFla84} Adler, R., Flatto, L.: The backward continued fraction map and geodesic flow.
%   Ergodic Theory and Dynamical Systems {\it  4} (1984) 487--492.
   

\bibitem{BaRa60} Bahadur, R. R.,  Ranga Rao, R.:
On deviations of the sample mean. Ann. Math. Stat. {\bf 31},  1015--1027 (1960)

% \bibitem{Bea71} A. F. Beardon. Inequalities for certain Fuchsian Groups. Acta Math. {\bf  127} (1971), 221--258. \textcolor{red}{Not referred.}

 \bibitem{Bea83}Beardon,  A. F.: {\it The geometry of discrete groups.} Graduate Texts in Mathematics {\bf  91} (1983)

 \bibitem{BeaMask74} Beardon, A. F., Maskit, B.: Limit points of Kleinian groups and finite sided fundamental              polyhedra. Acta Math. {\bf  132}, 1--12 (1971)
 
%\bibitem{BS84}Birman, J. \& Series, C. Dehn's algorithm revisited, with applications to simple curves on surfaces. {\em Combinatorial Group Theory And Topology (Alta, Utah, 1984)}. \textbf{111} pp. 451-478 (1987) 
 

%\bibitem{Bes34} Besicovitch, A.:
%On the sum of digits of real numbers represented in the dyadic system. Math. Ann. {\it  110} (1934) 321--330.

%Besicovitch, A. S.: Sets of fractal dimension IV: On rational approximation to real numbers.
%J. London Math. Soc. {\it  9} (1934) 126--131.

%\bibitem{BisJon97} Bishop, C. J., Jones P. W.:
%Hausdorff dimension and Kleinian groups. Acta Math. {\it  179} (1997) 1--39.


    
   

 \bibitem{Bow75} Bowen, R.:
 {\it Equilibrium states and the ergodic theory of Anosov
  diffeomorphisms,} Second revised edition.
  Lecture Notes in Mathematics, {\bf 470}
   Springer-Verlag, Berlin 2008.

%\bibitem{Burger93}Burger, M. Intersection, the Manhattan curve, and Patterson-Sullivan theory in rank 2. {\em Internat. Math. Res. Notices}., 217-225 (1993).
 
   
   
 

 \bibitem{BowSer79} Bowen, R.,  Series, C.: Markov maps associated with fuchsian groups.
 Inst. Hautes \'Etudes Sci. Publ. Math.  {\bf  50}, 153--170 (1979)

\bibitem{bridsonhaefliger}Bridson, M., Haefliger, A.: Metric spaces of non-positive curvature. Springer, 1999. 
 
  \bibitem{ChaCol05} Chazottes, J.-R.,  Collet, P.:  Almost-sure central limit theorems
and the Erd\H{o}s-R\'enyi law for expanding maps of the interval.
Ergod. Th. $\&$ Dynam. Sys. {\bf  25}, 419--441 (2005)

\bibitem{DenNic13} Denker, M., Nicol, M.: Erd\"os-R\'enyi laws for dynamical systems. J. London Math. Soc. {\bf 87}, 497--508 (2013)

% \bibitem{Cli13} Climenhaga, V.: Topological pressure of simultaneous level sets.
 %Nonlinearity {\it  26} (2013) 241--268.
 
  \bibitem{eizen94}Eizenberg, A.,  Kifer, Y., Weiss, B.: Large deviations for $\mathbb Z^d$-actions. Commun. Math. Phys. {\bf 164}, 433-454 (1994)
\bibitem{Ell85} Ellis, R. S.: 
 {\it Entropy, large deviations, and statistical mechanics}, 
   {\it Grundlehren der Mathematischen Wissenschaften} {\bf  271}, Springer (1985)
 
 %\bibitem{Fal} Falconer, K.: {\it Fractal Geometry}. Mathematical foundations and
% applications. Second edition. John Wiley $\&$ Sons, 2003

%\bibitem{Egg49} Eggleston, H., G.: The fractional dimension of a set defined by decimal properties.
%Quart. J. Math., Oxford Ser. {\it  20} (1949) 31--36.

%Sets of fractional dimension which occur in some problems of number theory.
%Proc. Lond. Math. Soc. {\it  54} (1952), 42--93.



 
% \bibitem{FanJorLiaRam16} Fan, A.-H.,  Jordan, T., Liao, L., Rams, M.:
%Multifractal analysis for expanding interval maps with infinitely many branches.
%Trans. Amer. Math. Soc. {\it  367} (2015) 1847--1870.

%\bibitem{FLM10} Fan, A.-H., Liao, L.-M., Ma, J.-H.: On the frequency of partial quotients of regular continued fractions. Math. Proc. Cambridge Philos. Soc. {\it  148} (2010) 179--192. 

%\bibitem{FLMW10} Fan, A.-H., Liao, L.-M., Ma, J.-H., Wang, B.-W.: Dimension of Besicovitch-Eggleston
%sets in countable symbolic space. Nonlinearity {\it  23} (2010) 1185--1197.


%\bibitem{FieFieYur02} Fiebig, D., Fiebig, U.-R., Yuri, M.:
%Pressure and equilibrium states for countable state Markov shifts.
%Israel J. Math. {\it  131}, 221--257 (2002)

%\bibitem{FTZ21}L. Fang, H. Takahasi, Y. Zhang. Precise asymptotics on the Birkhoff sums for dynamical systems,  Nonlinearity {\bf 34} (2021) 7095--7108.

\bibitem{floyd80}
Floyd, W. J.:  Group completions and limit sets of Kleinian groups. Invent. Math. {\bf  57}, 205--218
(1980)


%\bibitem{GelRam09} K. Gelfert, M. Rams. The Lyapunov spectrum of some parabolic systems. Ergod. Th. $\&$ Dynam. Sys.  {\bf 29} (2009), 919--940.



%\bibitem{HN} N. Haydn, M. Nicol: Erd\"os-R\'enyi laws for exponentially and polynomially mixing dynamical systems. arXiv. 2103.00645

%\bibitem{Iom10} G. Iommi. Multifractal analysis of the Lyapunov exponent for the backward continued fraction map. Ergod. Th. $\&$ Dynam. Sys.  {\bf  30} (2010), 211-232.



%\bibitem{Hof95} Hofbauer, F.: Local dimension of piecewise monotone maps on the interval.
%Ergod. Th. Dynam. Sys. {\it  15} (1995) 1119--1142.

%\bibitem{Hof10} Hofbauer, F.: Multifractal
%spectra of Birkhoff averages for piecewise monotone interval map.
%Fund. Math. {\it  208} (2010) 95--121.

%\bibitem{HofRai92} Hofbauer, F., Raith, P.: The Hausdorff dimension of an ergodic invariant measure
%for a piecewise monotonic map of the interval.
%Canad. Math. Bull. {\it  35} (1992) 84--98.





%\bibitem{IomJor} Iommi, G., Jordan, T.: Multifractal analysis of Birkhoff averages for
%countable Markov maps. Ergodic Theory and Dynamical Systems
%{\it  35} (2015) 2559--2586.

%\bibitem{IosKra02} Iosifescu, M., Kraaikamp, C.: {\it Metrical theory of continued fractions.}
%Mathematics and its Applications, {\it  547.} Kluwer Academic Publishers, Dordrecht, 2002 



%\bibitem{JaeKes10} Jaerisch, J., Kesseb\"ohmer, M.: Regularity of multifractal spectra
%of conformal iterated function systems. Trans. Amer. Math. Soc. {\it  363} (2011) 313--330.
%\bibitem{JKM21} J. Jaerisch, S. Munday, M. Kesseb\"ohmer. A multifractal analysis for cuspidal windings on hyperbolic surfaces. Stochastics and Dynamics {\bf  3} (2021), 2140007


%\bibitem{JaeTak20} J. Jaerisch, H. Takahasi. Mixed multifractal spectra of Birkhoff averages for non-uniformly expanding one-dimensional Markov maps with countably many branches. Adv. Math. {\bf  385} (2021), 107778 

\bibitem{JaeTak22} Jaerisch, J.,  Takahasi, H.:
Multifractal analysis of homological growth rates for hyperbolic surfaces. arXiv:2204.08907

%\bibitem{JJOP10} A. Johansson, T. Jordan, A. \"Oberg and M. Pollicott.  Multifractal analysis of non-uniformly hyperbolic systems. Israel J. Math. {\bf  177} (2010), 125--144.

%\bibitem{JorRam11} T. Jordan, M. Rams. Multifractal analysis of weak Gibbs measures for non-uniformly expanding $C^1$ maps. Ergod. Th. $\&$ Dynam. Sys.  {\bf 31} (2011), 143--164.

%\bibitem{JorRam} Jordan, T., Rams, M.: Birkhoff spectrum for piecewise monotone interval maps.
%arXiv:1712.03750 \textcolor{red}{Still a preprint. To be deleted.}



%\bibitem{Kao2020}Kao, L. Manhattan curves for hyperbolic surfaces with cusps. {\em Ergodic Theory Dynam. Systems}. \textbf{40}, 1843-1874 (2020).

\bibitem{Kes99} Kesseb\"ohmer,
M.:
Multifraktale und Asymptotiken grosser Deviationen,
Dissertation, Georg-August-Universit\"at zu G\"ottingen, 1999

\bibitem{KesStr04} Kesseb\"ohmer, M., Stratmann, B. O.:
A Multifractal formalism for growth rates and applications to geometrically finite
Kleinian groups. Ergod. Th. $\&$ Dynam. Sys. {\bf  24}, 141--170
(2004)

\bibitem{KesStr07} Kesseb\"ohmer, M., Stratmann, B. O.:
A multifractal analysis for Stern-Brocot intervals,
continued fractions and Diophantine growth rates.
J. Reine Angew. Math. {\bf 605}, 133--163 (2007)


\bibitem{Kif90}
Kifer, Y.:
Large deviations in dynamical systems and stochastic processes.
  Trans. Amer. Math. Soc. {\bf  321},  505--524 (1990)
 

%\bibitem{Ko29} Koebe, P.: Riemannische Manigfaltigkeiten und nichteuklidische Raumformen,
% IV. Sitzungberichte der Preussichen Akad. der Wissenschaften (1929) 414--457.




%\bibitem{Led81} F. Ledrappier: Some relations between dimension and Lyapunov exponents. Commun. Math. Phys. {\it  81} (1981) 229--238.

% \bibitem{Led81b} Ledrappier, F.:
%Some properties of absolutely continuous invariant measures on an interval.
%Ergod. Th. $\&$ Dynam. Sys.  {\it  1} (1981) 77--93.

%\bibitem{Mak99} Makarov, N. G.: Fine structure
%of harmonic measure. St. Petersburg Math. J.
%{\it  10} (1999) 217--268.

%\bibitem{MauUrb96}  Mauldin, R.D., Urba\'nski, M.: Dimension and measures in infinite iterated %function systems.
%Proc. London Math. Soc. {\it  73} (1996) 105--154.
%\textcolor{red}{Not on parabolic IFSs. To be removed.}

%\bibitem{MauUrb00}  R. D. Mauldin, M. Urba\'nski. Parabolic iterated function systems. Ergod. Th. $\&$ Dynam. Sys.  {\bf 20} (2000), 1423--1447.
  
% \bibitem{MauUrb03} R. D. Mauldin, M. Urba\'nski. {\it Graph directed Markov systems: Geometry and Dynamics of Limit Sets.} Cambridge Tracts in Mathematics  {\bf  148} Cambridge University Press (2003)


 
%\bibitem{Mun12} Munday, S.: {\it On {H}ausdorff dimension and cusp excursions for {F}uchsian groups}, Discrete and Continuous Dynamical Systems. Series A, {\it  32} (2012), 2503--2520. 
 
 

%\bibitem{Mor97} Morita, T.: Markov systems and transfer operators associated with cofinite
%Fuchsian groups. Ergodic Theory and Dynamical Systems (1997) {\it  17} 1147--1181.

%\bibitem{Nak00} K. Nakaishi.  Multifractal formalism for some parabolic maps. Ergod. Th. $\&$ Dynam. Sys.  {\bf  20} (2000), 843--857.

%\bibitem{Nic89} Nicholls, P. J.: The ergodic theory of discrete groups. London Mathematical Society Lecture Note Series. {\it  143} (1989), Cambridge University Press, Cambridge.
   
 
% \bibitem{Ols03} L. Olsen.  Multifractal analysis of divergence points of deformed measure theoretical Birkhoff averages. J. Math. Pures Appl. {\bf  82} (2003), 1591--1649. 
  
%    \bibitem{Ols04} Olsen, L.: On the Hausdorff dimension of generalized Besicovitch-Eggleston sets of $d$-tuples of numbers.
%  Indag. Math. {\it  15} (2004) 535--547.
  
\bibitem{OrePel89}
Orey, S., Pelikan, S.: Deviations of trajectory averages and the defect in Pesin's formula
  for Anosov diffeomorphisms.
 Trans. Amer. Math. Soc. {\bf  315}, 741--753 (1989) 

  
%  \bibitem{Pat76} S. J. Patterson.  The limit set of a Fuchsian group. Acta Math. {\bf  136} (1976), 241--273. \textcolor{red}{Not referred.}

 
%  \bibitem{Pes97} Y. Pesin.  {\it Dimension Theory in Dynamical Systems: Contemporary Views and Applications}, The University of Chicago Press, 1997.
 
 
 


%\bibitem{PesWei97} Y. Pesin, H. Weiss. A multifractal analysis of equilibrium measures for conformal expanding maps and Moran-like geometric constructions. J. Stat. Phys. {\bf  86} (1997), 233--275.


%\bibitem{PesWei01} Y. Pesin, H. Weiss. The multifractal analysis of Birkhoff averages and large deviations in Global Analysis of Dynamical Systems, Institute of Physics, Bristol, (2001), 419--431.

%\bibitem{PS2016} Pollicott, M. \& Sharp, R. Weil-Petersson metrics, Manhattan curves and Hausdorff dimension. {\em Math. Z.}. \textbf{282}, 1007-1016 (2016).

%\bibitem{Pol91} Pollicott, M.: Some applications of thermodynamic formalism to manifolds
%with constant negative curvature. Adv. Math. {\it  85} (1991) 161--192. 

 

\bibitem{PolWei99} Pollicott, M.,  Weiss, H.: Multifractal analysis of Lyapunov exponent for  continued fraction and Manneville-Pomeau transformations and applications to Diophantine Approximation. Commun. Math. Phys. {\bf  207}, 145--171
(1999).


 


   \bibitem{Rat94} Ratcliffe,
   J. G.: {\it Foundations of hyperbolic manifolds}, Graduate Texts in Mathematics {\bf  149}  (1994)
   
%   \bibitem{Rue78b} Ruelle, D.: An inequality for the entropy of differentiable maps. Bol Soc. Brasil Mat. {\it  9} (1978) 83--87.
      
   \bibitem{Rue04} Ruelle, D.: {\it Thermodynamic formalism. The mathematical structures of classical equilibrium statistical mechanics.} 
Second edition. Cambridge University Press (2004)



%\bibitem{Ser81} Series, C.: Symbolic dynamics for geodesic flows.
%Acta Math. {\it  146} (1981), 103--128.

\bibitem{Ser81b} Series, C.: The infinite word problem and limit sets in Fuchsian groups.
Ergod. Th. $\&$ Dynam. Sys.  {\bf  1}, 337--360 (1981)

\bibitem{Ser86} Series, C.: Geometrical Markov coding of geodesics on surfaces of constant negative curvature.
Ergod. Th. $\&$ Dynam. Sys.  {\bf  6}, 601--625 (1986)

%\bibitem{Sul79} D. Sullivan. The density at infinity of a discrete group of hyperbolic motions. Inst. Hautes \'Etudes Sci. Publ. Math. {\bf  50} (1979), 171--202. \textcolor{red}{Not referred.}

%\bibitem{Sul84} D. Sullivan, Entropy, Hausdorff measures old and new, and limit sets of geometrically finite Kleinian groups. Acta Math. {\it  153} (1984) 259--277.

\bibitem{Tak84}
Takahashi, Y.:
 Entropy functional (free energy) for dynamical systems and their
  random perturbations.
 In {\it Stochastic analysis (Katata/Kyoto, 1982)}, North-Holland Math. Library, {\bf  32},  437--467 (1984) North-Holland, Amsterdam
   
  \bibitem{T20}
Takahasi, H.:
Large deviations for denominators of continued fractions.
      Nonlinearity {\bf 33} 5861--5874 (2020) 
 
  


   
%   \bibitem{TakVer03} Takens, F., Verbitskiy, E.:
 %  On the variational principle for the topological entropy of certain non-compact sets.
  % Ergodic Theory and Dynamical Systems {\it  23} (2003) 317--348.
   
% \bibitem{Tha83} Thaler, M.: Transformations on $[0,1]$ with infinite invariant measures. Isr. J. Math. {\bf 46}, 67--96 (1983) 
   
  % \bibitem{Urb96} Urba\'nski, M.: Parabolic Cantor sets. 
   %Fundamenta Mathematicae {\it  151} (1996) 241--277.

% \bibitem{Wad96} Waddington S 1996 Large deviation asymptotics for Anosov flows.
%  Ann. Inst. Henri Poincar\'e {\it  13} 445--484.
   
%   \bibitem{Yur94} Yuri, M.: Invariant measures for certain multi-dimensional maps. Nonlinearity {\it  7} (1994) 1093--1124.



%\bibitem{Wei99} H. Weiss. The Lyapunov spectrum for conformal expanding maps and Axiom A surface diffeomorphisms. J. Stat. Phys. {\bf  95} (1999), 615--632.



 
\bibitem{You90} Young, L.-S.: Some large deviation results for dynamical systems. Trans. Amer. Math. Soc. {\bf 318}, 525--543 (1990)




%\bibitem{Yur02} M. Yuri. Multifractal analysis of weak Gibbs measures for intermittent systems. Commun. Math. Phys. {\bf  230} (2002), 365--388.
   
%   \bibitem{Zwe05} Zweim\"uller, R.: Invariant measures for general(ized) induced transformations.
%Proc. Amer. Math. Soc. {\it  133} (2005) 2283--2295.

   \end{thebibliography}
\end{document}